\documentclass[a4paper]{amsart} 
\usepackage[a4paper]{geometry}
\usepackage[utf8]{inputenc}
\usepackage{babel}
\usepackage{amsthm,amsmath,amssymb,amsfonts,amscd}
\usepackage{mathtools}
\usepackage{hyperref}
\usepackage{xcolor}
\usepackage{subfiles}
\usepackage{enumerate}
\usepackage{comment}

\newtheorem{theorem}{Theorem}[section]
\newtheorem{lemma}[theorem]{Lemma}
\newtheorem{prop}[theorem]{Proposition}
\theoremstyle{definition}
\newtheorem{definition}{Definition}
\theoremstyle{remark}
\newtheorem{remark}{Remark}

\newcommand{\R}{\ensuremath{\mathbb{R}}}

\newcommand{\Lipa}[1][\gamma]{\ensuremath{\mathcal{C}^#1}} 
\newcommand{\Na}[2][\gamma]{\ensuremath{\left\|#2\right\|_{#1}}} 
\newcommand{\SNa}[2][\gamma]{\ensuremath{\left|#2\right|_{#1}}} 
\newcommand{\Np}[2][\infty]{\ensuremath{\left\|#2\right\|_{L^#1}}}
\newcommand{\NUa}[2][\gamma]{\ensuremath{\left\|#2\right\|_{1,#1}}} 
\newcommand{\LipUa}[1][\gamma]{\ensuremath{\Lipa[{{1,\gamma}}]}} 
\newcommand{\LPet}[1][\gamma]{\ensuremath{c^{#1}}} 
\newcommand{\LUPet}[1][\gamma]{\ensuremath{c^{1,#1}}} 

\newcommand{\SU}{\ensuremath{S^{(1)}}}
\newcommand{\SD}{\ensuremath{S^{(2)}}}

\newcommand{\N}[1][\alpha]{\ensuremath{\mathcal{N}_{#1}}}
\newcommand{\F}[1][\alpha]{\ensuremath{\mathcal{F}_{#1}}}

\usepackage{xcolor}

\author{Marc Magaña}
\address{Departament de Matemàtiques -- Universitat Autònoma de Barcelona\\
08193 Bellaterra, Barcelona, Spain}
\email{marc.magana@uab.cat}

\author{Joan Mateu}
\address{Departament de Matemàtiques -- Universitat Autònoma de Barcelona\\
08193 Bellaterra, Barcelona, Spain}

\address{Centre de Recerca Matemàtica\\
08193 Bellaterra, Barcelona, Spain}

\email{joan.mateu@uab.cat}

\author{Joan Orobitg}
\address{Departament de Matemàtiques -- Universitat Autònoma de Barcelona\\
08193 Bellaterra, Barcelona, Spain}
\email{joan.orobitg@uab.cat}

\title[Vortex Patch Regularity for QGSW]{The regularity of the boundary of vortex patches for the quasi-geostrophic shallow-water equations}

\begin{document}
\begin{abstract}
We prove the persistence of boundary smoothness of vortex patches for the quasi-geostrophic shallow-water (QGSW) equations. The QGSW equations generalize the Euler equations by including an additional parameter, the Rossby radius $\varepsilon^{-1}$, which modifies the relationship between the streamfunction and the (potential) vorticity. In addition, we prove that solutions of the QGSW equations converge locally in time to the corresponding Euler solutions as $\varepsilon \to 0$ in little Hölder spaces.
\end{abstract}

\maketitle

\tableofcontents
\section{Introduction}
The quasi-geostrophic shallow-water equations, commonly used to track the dynamics of the atmospheric and oceanic circulation at large-scale motion, is a two-dimensional (2D) active scalar equation given by 
\begin{equation*}\label{def:QGSW}\tag{QGSW}
    \left\{\begin{aligned}
        &\partial_tq + v\cdot \nabla q=0, \quad (t,x)\in \R^+\times \R^2\\ 
        &v= \nabla^\perp(\Delta-\varepsilon^2)^{-1}q, \\
        &q(0,\cdot)=q_0,
    \end{aligned}\right.
\end{equation*}
where $\nabla^\perp= (-\partial_2, \partial_1)^t$ denotes the orthogonal gradient. The quantities involved are the divergence-free velocity field $v$, the potential vorticity $q$, which is a scalar function, and a real parameter $\varepsilon$. In the literature, when $\varepsilon>0$ it is known as the inverse Rossby radius and is defined by $$\varepsilon=\frac{\omega_c}{\sqrt{gH}},$$ where $g$ is the gravitational constant, $H$ is the mean active layer depth and $\omega_c$ is the Coriolis frequency, assumed to be constant. 

These equations are derived asymptotically from the rotating shallow water equations, in the limit of rapid rotation and weak variations
of the free surface. For a comprehensive overview, see \cite{plotka2011shallow,vallis2017atmospheric} and the references therein. For $\varepsilon=0,$ the equations reduce to the well-known vorticity formulation of the 2D Euler equation, for which numerous results are known. 

To give some examples, Yudovich \cite{yudovich} proved that the 2D Euler vorticity equation is well-posed in $L^\infty_c$, the space of bounded measurable functions with compact support. A vortex patch is a special weak solution of this equation when the initial condition is the characteristic function of a bounded domain $\Omega_0$. Since the vorticity equation is a transport equation, vorticity is conserved along trajectories, meaning that at any time $t$, it remains the characteristic function of some domain $\Omega_t$. 

A significant problem posed in the 1980s was whether the boundary smoothness of the initial domain persists over time. Specifically, if $\Omega_0$ has a boundary with derivatives in the Hölder class, that is $\partial\Omega_0\in\LipUa$ for $0 < \gamma < 1$, one would like $\Omega_t$ to maintain a boundary of the same class indefinitely. Chemin \cite{chemin1993persistance} proved that boundary regularity persists for all times using paradifferential calculus. Shortly after, Bertozzi and Constantin \cite{bertozzi1993global} provided an alternative proof based on methods of classical analysis with a geometric approach.

Recently, \eqref{def:QGSW} equations have garnered interest in the mathematical community, as evidenced by works such as \cite{renault,hmidi2021time,roulley2023vortex,YAN20241}. However, the specific aspects of solution regularity and vortex patch regularity, which have been thoroughly studied for the 2D Euler equations, had not yet been explicitly addressed for the \eqref{def:QGSW} equations. Our paper delves into these issues, exploring the regularity of solutions and vortex patches within the \eqref{def:QGSW} framework.

From now on, assume \(\varepsilon > 0\) for simplicity. The results remain valid for \(\varepsilon < 0\) by replacing \(\varepsilon\) with \(|\varepsilon|\). By using that the fundamental solution of the operator $(\Delta-\varepsilon^2)$ is $-\frac{1}{2\pi}K_0(\varepsilon|x|),$ one can write the velocity $v$ in the \eqref{def:QGSW} as \begin{equation}\label{def:kernel}
    v = K*q, \text{ with } K(x)=\frac{\varepsilon}{2\pi}\frac{x^\perp}{|x|} K_1(\varepsilon|x|),
\end{equation}
where $K_0$ and  $K_1$ are modified Bessel functions, which we will introduce in Section \ref{sec:preliminaries} (see \cite{renault} for the details). 

The main results of this paper are the following. We first 
establish a complete well-posedness theory for the \eqref{def:QGSW} equation: existence and uniqueness of solutions for initial data $q_0 \in \Lipa_c$, together with a Yudovich-type theory guaranteeing global well-posedness in $L^\infty_c$. Building on this, we prove the following vortex patch result.

 \begin{theorem}\label{th:vortex-patch}
         Let $\Omega_0$ be a bounded domain with boundary of class $C^{1,\gamma},\;0<\gamma<1.$ Then the 
         \begin{equation*}\tag{QGSW}
    \left\{\begin{aligned}
        &\partial_tq + v\cdot \nabla q=0, \quad (t,x)\in \mathbb{R}^+\times \mathbb{R}^2\\ 
        &v= K * q, \\
        &q(0,\cdot)=\chi_{\Omega_0},
    \end{aligned}\right.
\end{equation*} has a unique weak solution of the form $q(x,t)=\chi_{\Omega_t}(x),$ with $\Omega_t$ a bounded domain with boundary of class $C^{1,\gamma}.$ 
    \end{theorem}   

{It is worth noting that the persistence of boundary smoothness of vortex patches has been studied for a broad families of nonlinear transport equations} like the one above, in which the velocity is given by a convolution of the scalar function with a kernel, $K$, for different kernels. For instance, when $K=\nabla N,$ where $N$ is the fundamental solution of the Laplacian one has the Aggregation patch problem, which was studied in \cite{bertozzi2012aggregation,bertozzi2016regularity}. More recently, Cantero, Mateu, Orobitg and Verdera in \cite{cantero2021regularity} extended this result to a broader class of kernels that are smooth away from the origin, homogeneous and odd, and later revisited 
by Verdera \cite{VerderaRevisited}, who provided a different approach in the two-dimensional setting.

Note that the kernel in \eqref{def:kernel} is non-homogeneous, and therefore falls outside the framework 
of \cite{cantero2021regularity,VerderaRevisited} for convolution models.  Our key idea to overcome this non-homogeneity is to decompose the gradient of the kernel into two parts: one with zero mean on spheres that behaves like a Calderón–Zygmund operator, and a much more regular (integrable) remainder. The first part gives rise to a commutator and admits exactly the same structure as in the homogeneous case, while the second part contributes only an additive constant to the logarithmic inequality (see Remark~\ref{remark:1derivades} and the discussion below).

{A second main contribution of this paper is a rigorous convergence result connecting \eqref{def:QGSW} to the 2D Euler equations. In Section \ref{sec:limit}, we study the behavior of solutions in the limit \(\varepsilon \to 0\). Theorem \ref{th:limitRegular} establishes that, for initial data compactly supported in the little Hölder space \(\LPet_c\) (the closure of $\mathcal{C}^\infty$ in the usual Hölder norm), the unique solutions of \eqref{def:QGSW} converge in the \(\Lipa\) norm to the corresponding Euler solution for the same initial data on a uniform time interval as $\varepsilon \to 0$. This provides a rigorous justification of the formal limit between the two models. 
}

The outline of this paper is as follows. In Section \ref{sec:preliminaries}, we introduce the modified Bessel $K_n$ functions, state some of their important properties, and recall the definitions of Hölder spaces. We also present relevant results concerning these spaces and how our kernel acts on them. In Section \ref{sec:regular} we prove existence and uniqueness of solutions for initial data, $q_0$, in the Hölder class $\Lipa,$ which we use later in Section \ref{sec:weak} to show existence and uniqueness of weak solutions. Later, in Section \ref{sec:patch} we focus on the regularity of the vortex patch boundary and prove Theorem \ref{th:vortex-patch}.{Finally, in Section \ref{sec:limit}, building on the well-posedness theory developed in Section \ref{sec:regular}, we prove Theorem~\ref{th:limitRegular}.}

{\textbf{A note on related work.} After completing this work, we were informed by one of the authors of the very recent job by Tan, Xue, and Xue \cite{TXX2025}, in which the authors study a broad family of active scalar equations. As noted in Remark 2 of that work, the \eqref{def:QGSW} equation fits within their framework, and they obtain results analogous to our Theorems \ref{th:vortex-patch} and \ref{th:yudovichQGSW}, although the methods are quite different. We thank them for bringing their work to our attention.}

\section{Preliminaries}\label{sec:preliminaries}
\subsection{Modified Bessel functions}
Solutions to $$z^2\frac{\partial^2 w}{\partial z^2} + z\frac{\partial w}{\partial z} - (z^2+\nu^2)w=0$$ are $I_{\pm\nu}(z)$ and $K_\nu(z),$ which are referred to as modified Bessel Functions. We shall collect some useful information about these {complex-valued} functions. The expressions below are taken from \cite[pp. 374--378]{abramowitz1988handbook}, unless otherwise noted. For further details, see \cite[Chapter 9]{abramowitz1988handbook} and \cite{watson1922treatise}. {For the sake of generality, we state the formulas in their standard complex form as in the cited references, although throughout this work we will only make use of them for real and positive values of~$z$.}

When $\nu=n\in \mathbb{Z}$ we have
\begin{equation}\label{def:besselK}
\begin{split}
    K_n(z)=\frac{1}{2}\left(\frac{z}{2}\right)^{-n}\sum_{k=0}^{n-1} \frac{(n-k-1)!}{k!}\left(\frac{-z^2}{4}\right)^{k} + (-1)^{n+1}\ln\left(\frac{z}{2}\right)I_n(z)\\
    +\frac{1}{2}\left(\frac{-z}{2}\right)^{n}\sum_{k=0}^\infty (\psi (k+1) + \psi(n+k+1))\frac{\left(\frac{z^2}{4}\right)^{k}}{k!(n+k)!},
\end{split}
\end{equation}
where $$I_n(z)=\sum_{m=0}^\infty \frac{\left(\frac{z}{2}\right)^{n+2m}}{m!(n+m)!},\; |\arg z |<\pi \text{ and } \psi(1)=-\gamma, \; \psi(m+1)=\sum_{k=1}^m\frac{1}{k}-\gamma,$$ being $\gamma=\lim_{j\to\infty}\left(\sum_{k=1}^j\frac{1}{k}-\ln(j)\right)$ Euler's constant. Note that $K_n(z)$ is singular at $z=0$ for every integer $n$.

In particular, 
\begin{equation}\label{def:besselK0}
    K_0(z)=-\log\left(\frac{z}{2}\right)I_0(z)+\sum_{m=0}^\infty \frac{\left(\frac{z}{2}\right)^{2m}}{(m!)^2}\psi(m+1), 
\end{equation}
so $K_0$ behaves like a logarithm at 0. For $n\geq 1$ one has the following asymptotic expansion \begin{equation}\label{propietat:besselKprop}
    K_n(z)\sim \frac{(n-1)!}{2\left(\frac{1}{2}z\right)^n} = 2^{n-1} (n-1)! z^{-n}, \quad \text{as}\quad z\to 0.
\end{equation}

For large $|z|$ we may use the following representation, which is equivalent to \eqref{def:besselK}:
\begin{equation}\label{def:besselKint}
    K_n(z)=\frac{\pi^\frac{1}{2}\left(\frac{1}{2}z\right)^n}{\Gamma\left(n+\frac{1}{2}\right)} \int_1^\infty e^{-zt}(t^2-1)^{n-\frac{1}{2}}dt,\quad n\geq 0,\; |\arg z | <\frac{1}{2}\pi,
\end{equation}
and the corresponding asymptotic expansion for large $|z|$ and fixed $n:$
\begin{equation}\label{propietat:besselKlluny}
    K_n(z)\sim \sqrt{\frac{\pi}{2z}} e^{-z} \left(1+\frac{4n^2-1}{8z}+\frac{(4n^2-1)(4n^2-9)}{2!(8z)^2}+\cdots\right),\quad |\arg z|<\frac{3}{2}\pi. 
\end{equation}

Moreover, for all $\alpha\in\R$ satisfying $\alpha > n$,  
\begin{equation}\label{eq:IntdefBessel}
    \int_0^\infty t^{\alpha-1} K_n(t)dt = 2^{\alpha-2}\Gamma \left(\frac{\alpha-n}{2}\right)\Gamma \left(\frac{\alpha+n}{2}\right),
\end{equation}
see \cite[p. 388]{watson1922treatise}. 

Finally, since derivatives of $K_n$ will be involved in our analysis, we recall the following recurrence relations:
\begin{align}
    K_n'(z)&=-K_{n-1}(z)-\frac{n}{z}K_n(z)=-K_{n+1}(z)+\frac{n}{z}K_n(z) \label{propietat:besselDerivades}\\
    \frac{2n}{z}K_n(z)&=K_{n+1}(z)-K_{n-1}(z) \label{propietat:besselRecurrencia}
\end{align}

\subsection{Inequalities for Hölder continuous functions and Kernel estimates}
We now recall the definition of Hölder spaces as well as some well-known properties of its elements.

\begin{definition} \label{def:holder}
Given $0<\gamma<1$ and $f:\R^n\to \R,$ set $$\Np{f}=\sup_{x\in \R^n} |f(x)| \text{ and } \SNa{f}=\sup_{\substack{x,y\in\R^n\\ x\neq y}} \frac{|f(x)-f(y)|}{|x-y|^\gamma}.$$
We define the norm $$\Na{f}\coloneqq \Np{f}+\SNa{f}.$$
For $$\begin{array}{rccl}
     F:&\R^n&\to& \R^d \\
     & x & \to & F(x)=(f_1(x),\dots,f_d(x)) 
\end{array}$$ we define $$\Na{F}\coloneqq \sup_{i=1,\dots, d} \Na{f_i},$$ and the Hölder space $$\Lipa (\R^n; \R^d) \coloneqq \left\{ f: \R^n\to \R^d:\Na{f}<\infty\right\}.$$
We also define $$\NUa{F}=\Np{F}+\Np{\nabla F} + \SNa{\nabla F},$$ and the Hölder space $\LipUa(\R^n;\R^d)$ as $$\LipUa(\R^n;\R^d)\coloneqq \left\{f: \R^n\to \R^d : \NUa{f} < \infty \right\}.$$
\end{definition}

\begin{lemma}\label{lemma:propietats}
Let $f,g\in\Lipa$ with $0<\gamma<1$. Then 
\begin{align}
    \SNa{fg}&\leq \Np{f}\SNa{g}+\SNa{f}\Np{g} \label{eq:SNprod} \\
    \Na{fg}&\leq \Na{f}\Na{g} \label{eq:algebra}
\end{align}
If moreover $X$ is a smooth invertible transformation in $\R^n$ satisfying 
$$|\det\nabla X(\alpha)|\geq c_1 > 0,$$
then there exists $c>0$ such that
\begin{align}
    \Na{(\nabla X)^{-1}}&\leq c\Na{\nabla X}^{2n-1} \label{eq:NaGradNabla}\\
    \NUa{X^{-1}}&\leq c \NUa{X}^{2n-1}\\
    \SNa{f\circ X}&\leq \SNa{f}\Np{\nabla X}^\gamma \label{eq:SNaComp}\\ 
    \Na{f\circ X}&\leq \Na{f}(1+\NUa{X}^\gamma )\\
    \Na{f\circ X^{-1}}&\leq \Na{f}(1+\NUa{X}^{\gamma(2n-1)}) \label{eq:NaCompInversa}
\end{align}
\end{lemma}
The proof of Lemma \ref{lemma:propietats} can be found in \cite[pp. 159 - 163]{MajdaBertozzi}.

\begin{definition}
    A function $f:\R^n\to \R$ is \emph{Log-Lipschitz} if there exists a constant $C>0$ such that $$|f(x)-f(y)|\leq C |x-y||\log{|x-y|}|.$$

    A function $f:\R^n\to \R^d$ is Log-Lipschitz if each of its components is.    
\end{definition}

We now state some properties of the kernel defined in \eqref{def:kernel} that will be used repeatedly throughout the proofs of the upcoming results.

\begin{lemma}\label{lemma:propietatsK}
    Let $K$ be as defined in \eqref{def:kernel}. Then, for all $x,y\in\R^2\setminus\{0\}$
    \begin{align}
        |K(x)|&\leq \frac{C}{|x|},\label{eq:DecK} \\
        |\nabla K(x)|&\leq \frac{C}{|x|^2},\label{eq:DecGradK} \\
        |K(x)-K(y)|&\leq \frac{C|x-y|}{|x||y|},\label{eq:difK}    
    \end{align}
    for some universal constant $C>0,$ that does not depend on $\varepsilon.$ 
\end{lemma}
\begin{proof}
    We show first that Bessel functions, $K_n$ for $n\geq 1$, satisfy \begin{equation} \label{eq:controlBessel}
        |K_n(r)|\leq \frac{C}{|r|^n}, \quad {\forall r\in (0,\infty),}
    \end{equation}     
    which will be used throughout. To do so we proceed as in \cite[Lemma 2.2]{ZinebHaroune}, the proofs of \eqref{eq:DecK} and \eqref{eq:difK} can also be found there. 

    {For $r\in(0,1]$, the series representation \eqref{def:besselK} shows that the leading term of $K_n(r)$ is proportional to $r^{-n}$, with all remaining terms involving higher powers of $r$. Thus, $K_n(r)$ can be regarded as a perturbation of $r^{-n}$, and in particular $|r^n K_n(r)| \leq C$.}
    
    For $r>1$ we can use that by the integral representation shown in \eqref{def:besselKint}
   { \begin{align*}      |K_n(r)|
    &= \left|\frac{\pi^\frac{1}{2}\left(\frac{1}{2}r\right)^n}{\Gamma\left(n+\frac{1}{2}\right)} \int_1^\infty e^{-rt}(t^2-1)^{\frac{1}{2}} (t^2-1)^{n-1}dt\right| \\
    &\leq \left|\frac{\pi^\frac{1}{2}\left(\frac{1}{2}r\right)^n}{\Gamma\left(n+\frac{1}{2}\right)} \int_1^\infty t^{2n-2}e^{-rt}(t^2-1)^{\frac{1}{2}}dt\right| \leq C_n e^{-r}, \\
    \end{align*}}
    where in the last step we use that $$\int_1^\infty t^s e^{-rt}(t^2-1)^{\frac{1}{2}} dt,$$ has an exponential decay at infinity for any $s\in\mathbb{N}.$ The proof of this follows by induction and can be seen in \cite[Lemma 2.2]{ZinebHaroune}.   
    
    We are now ready to prove \eqref{eq:DecGradK}. To do so, we compute the derivatives of the kernel:
    \begin{multline}
        \partial_1 K(x) = \frac{\varepsilon}{2\pi}\left(\frac{x_1 x_2}{|x|^3}K_1(\varepsilon|x|)-\frac{\varepsilon x_1 x_2}{|x|^2}K_1'(\varepsilon |x|),\right.\\\left. \frac{1}{|x|}K_1(\varepsilon|x|)-\frac{x_1^2}{|x|^3}K_1(\varepsilon |x|)+\frac{\varepsilon x_1^2}{|x|^2}K_1'(\varepsilon |x|) \right) \label{eq:KernelPartial1}
    \end{multline}
    \begin{multline}
        \partial_2 K(x) = \frac{\varepsilon}{2\pi}\left(-\frac{1}{|x|}K_1(\varepsilon|x|)+\frac{x_2^2}{|x|^3}K_1(\varepsilon |x|)-\frac{\varepsilon x_2^2}{|x|^2}K_1'(\varepsilon |x|) ,\right.\\\left. -\frac{x_1 x_2}{|x|^3}K_1(\varepsilon|x|)+\frac{\varepsilon x_1 x_2}{|x|^2}K_1'(\varepsilon |x|) \right)\label{eq:KernelPartial2}
    \end{multline}

    We focus on proving the result for the first component of $\partial_1 K,$ for the other components one can proceed analogously.

    Notice that by the second equality of \eqref{propietat:besselDerivades},
   \begin{align*}
    &\frac{\varepsilon}{2\pi}
    \left(\frac{x_1 x_2}{|x|^3}K_1(\varepsilon|x|)
    -\frac{\varepsilon x_1 x_2}{|x|^2}K_1'(\varepsilon|x|)\right) \\[3pt]
    &= \frac{\varepsilon}{2\pi}
    \left[\frac{x_1 x_2}{|x|^3}K_1(\varepsilon|x|)
    -\frac{\varepsilon x_1 x_2}{|x|^2}
    \left(-K_2(\varepsilon|x|)+\frac{1}{\varepsilon |x|}K_1(\varepsilon|x|)\right)\right] = \frac{\varepsilon^2}{2\pi}\frac{x_1 x_2}{|x|^2}K_2(\varepsilon|x|).
\end{align*}

Therefore, using \eqref{eq:controlBessel} and the factor $\varepsilon^2$ in front, we deduce that 
$|\partial_1 K(x)| \leq \frac{C}{|x|^2}$ with a constant $C$ independent of $\varepsilon$.
\end{proof}

\begin{remark}\label{remark:1derivades}
    Notice that the first component of $\partial_1 K$ as written on \eqref{eq:KernelPartial1} has zero mean over spheres and by using \eqref{propietat:besselDerivades} the second component can be rewritten as $$\frac{x_2^2-x_1^2}{|x|^3}K_1(\varepsilon |x|)-\frac{\varepsilon x_1^2}{|x|^2}K_0(\varepsilon |x|),$$ where the first term also has mean-value zero on spheres and has modulus bounded by $\frac{C}{|x|^2}$ (by \eqref{eq:controlBessel}) and the second term is integrable by \eqref{eq:IntdefBessel}. The analogous happens for $\partial_2K$, so we can write {
    \begin{align*} 
    \nabla K&\coloneqq \SU+\SD\\&=\frac{1}{2\pi}\left[\begin{pmatrix}
        \frac{\varepsilon^2 x_1 x_2}{|x|^2} K_2(\varepsilon |x|) & \frac{\varepsilon(x_2^2-x_1^2)}{|x|^3}K_1(\varepsilon |x|)\\
         \frac{\varepsilon(x_2^2-x_1^2)}{|x|^3}K_1(\varepsilon |x|)& -\frac{\varepsilon^2 x_1 x_2}{|x|^2} K_2(\varepsilon |x|)
    \end{pmatrix}+\begin{pmatrix}
        0 & -\frac{\varepsilon^2 x_1^2}{|x|^2}K_0(\varepsilon |x|) \\
        \frac{\varepsilon^2 x_2^2}{|x|^2}K_0(\varepsilon |x|) & 0
    \end{pmatrix}\right]
    \end{align*}} where each component of $\SU$ has mean-value zero on spheres and $|\SU_{i,j}(x)|\leq \frac{C}{|x|^2}$, {while each component of $\SD$ belongs to $L^1(\mathbb{R}^2)$ by 
\eqref{eq:IntdefBessel}}
\end{remark}

\begin{lemma}\label{lemma:SIO}
    Let $K$ be the kernel defined in \eqref{def:kernel}. Let $P=\partial_iK, i=1,2.$ Set $$Tf(x)=\int_{\R^2} K(x-x')f(x')dx' \text{ and } Sf(x)= p.v. \int_{\R^2} P(x-x')f(x')dx,$$ where $p.v.$ stands for the Cauchy principal value. 

    For $0<\gamma < 1$ let $f\in \Lipa_c(\R^2;\R).$ Set $R^2\coloneqq m(supp(f))<\infty.$ Then, there exists a constant $c,$ independent of $\varepsilon$, $f$ and $R,$ such that 
    \begin{align}
        \Np{Tf}&\leq cR \Np{f}\label{eq:Ksna}\\
        \Np{Sf}&\leq c\left\{\SNa{f}\delta^\gamma +\max\left(1,\ln{\frac{R}{\delta}}\right)\Np{f}\right\}, \quad \forall \delta>0,\label{eq:SIOepsilon}\\
        \SNa{Sf}&\leq c \SNa{f}.    \label{eq:SIOsna}     
    \end{align}
\end{lemma}

\begin{proof}
    The proof of \eqref{eq:Ksna} is as in \cite[Lemma~4.5]{MajdaBertozzi}. To obtain \eqref{eq:SIOepsilon} and \eqref{eq:SIOsna}, we follow the strategy of \cite[Lemma~4.6]{MajdaBertozzi}, {which relies crucially on two properties of the kernel: having zero mean over spheres and exhibiting suitable gradient decay. The decomposition introduced in the previous remark isolates the part $\SU$, whose components have zero mean over spheres, while the remaining term $\SD$ is more regular. Thus, in order to apply the Majda--Bertozzi argument to $\SU$, it only remains to verify that $|\nabla \SU(x)| \le \frac{C}{|x|^{3}}.$}

 To do so, note that each component can be written as {$\SU_{i,j}=\varepsilon^m H^{(m-2)}(x)K_{m}(\varepsilon |x|)$} for $m=1,2$ where $H^{(m-2)}$ is a homogeneous function of degree $m-2$ such that $\nabla H^{(m-2)}$ is homogeneous of degree $m-3$. Therefore, when taking derivatives one gets {
    \begin{align*}
       |\partial_{x_k} \SU |&=\varepsilon^m \left| \partial_{x_k} H^{(m-2)}(x) K_m(\varepsilon|x|)+H^{(m-2)}(x)\frac{\varepsilon x_k}{|x|}K_m'(\varepsilon |x|)\right|\\&= \varepsilon^m\left| \partial_{x_k} H^{(m-2)}(x) K_m(\varepsilon|x|)+H^{(m-2)}(x)\frac{\varepsilon x_k}{|x|}\left(-K_{m+1}(\varepsilon|x|)+\frac{m}{\varepsilon|x|}K_m(\varepsilon|x|)\right)\right|\\&\leq \frac{C}{|x|^3},
    \end{align*}}
    
    where in the first equality we use the second part of \eqref{propietat:besselDerivades} and in the last step we use the homogeneity of $H$ and \eqref{eq:controlBessel}. {In particular, the bound \eqref{eq:controlBessel} implies that $C$ does not depend on $\varepsilon.$} 
    
    For $\SD$,{both \eqref{eq:SIOepsilon} and \eqref{eq:SIOsna} follow directly}. However, we show the proofs here for one of its components to stress the $\varepsilon$ independence of the constants:
    \begin{align*}
        |{\SD_{1,2}} f(x)| &= \left|\frac{\varepsilon}{2\pi}\int_{\R^2}\varepsilon\frac{(x_1-x_1')^2}{|x-x'|^2}K_0(\varepsilon |x-x'|) f(x') dx'\right|  
        \leq \varepsilon \Np{f} \int_0^\infty |\varepsilon rK_0(\varepsilon r) | dr\\ 
        &=\Np{f} \left|\int_0^\infty r'K_0(r')  dr'\right|\leq c \Np{f},
    \end{align*}
    where we have used \eqref{eq:IntdefBessel} and the fact that $K_0$ is of constant sign for $r>0.$ This shows \eqref{eq:SIOepsilon}. Finally, \eqref{eq:SIOsna} is proven using the same approach
    \begin{align*}
        |\SD_{1,2} f(x+h)- \SD_{1,2} f(x)| &= \left|\frac{\varepsilon}{2\pi}\int_{\R^2}\varepsilon\frac{y_1^2}{|y|^2}K_0(\varepsilon|y|)[f(x+h-y)-f(x-y)] dy \right| \\
        &\leq \SNa{f} |h|^\gamma \varepsilon \int_{0}^\infty |\varepsilon r K_0(\varepsilon r)| dr \leq c \SNa{f} {|h|^\gamma}. 
    \end{align*}
    
\end{proof}

\section{Existence and uniqueness of regular solutions}\label{sec:regular}

Our first goal is to establish the following well-posedness result for \eqref{def:QGSW} in the space of Hölder continuous functions.
\begin{theorem}\label{th:EUholder}
    For $0<\gamma<1,$ if $q_0\in \Lipa_c(\R^2;\R),$ then the \eqref{def:QGSW} has a unique solution $q(\cdot,t)\in \Lipa_c(\R^2;\R)$ for all time $t\in\R$. 
\end{theorem}

In order to prove theorem \ref{th:EUholder}, we follow the approach of \cite[Chapter 4]{MajdaBertozzi} for the Euler equation. That is, we prove existence and uniqueness by means of the particle–trajectory method. {We first establish existence and uniqueness for the 
flow map, and then deduce the corresponding results for \eqref{def:QGSW}.}

Let us first recall some definitions. Given a velocity $v$ and a point $\alpha\in\R^2$ we set, whenever it is well defined, the \emph{flow map} \begin{equation*}
\begin{aligned}
     X(\alpha,\cdot):\R&\to \R^2,\\ t&\to X(\alpha,t)
\end{aligned}
\end{equation*}
as the solution of the ordinary differential equation (ODE)
\begin{equation}
    \label{eq:ParticleTrajectories}
    \left\{\begin{array}{l}
          \frac{d}{dt} X(\alpha,t)=v(X(\alpha,t),t), \\
         X(\alpha,0)=\alpha.
    \end{array}\right.
\end{equation}
In integral form, the flow map is given by $$X(\alpha,t)=\alpha+\int_0^t v(X(\alpha,{s}),s)ds.$$

This map gives the position at time $t$ of the particle initially located at $\alpha$, moving under the velocity field $v$. It is also called the \emph{trajectory} of the particle initially at $\alpha$. If $q(\cdot, t)$ is smooth enough, then it is easy to check that $q(x,t)=q_0(X^{-1}(x,t))$ where, as usual, $X^{-1}(\cdot,t)$ denotes the inverse of the flow $X$. Hence, by \eqref{eq:NaCompInversa}, we have $q(\cdot,t)\in\Lipa$ whenever $X(\cdot,t)\in\LipUa$.

We focus on proving existence, uniqueness and regularity for $X$. To do so, we combine the fact that $X$ satisfies \eqref{eq:ParticleTrajectories} with $v(\cdot,t)=K*q(\cdot,t)$ at any given time to get $$\frac{dX}{dt}(\alpha,t)=\int_{\R^2} K(X(\alpha,t)-x')q(x',t)dx'.$$ 

Applying the change of variables $x'=X(\alpha',t)$, using that $v$ is divergence–free (so the Jacobian determinant equals one) and that $q$ is conserved along the flow, we obtain
\begin{align}
    \frac{dX}{dt}(\alpha,t)&=\int_{\R^2} K(X(\alpha,t)-X(\alpha',t))q_0(\alpha')d\alpha'\eqqcolon F(X(\alpha,t)),\label{eq:ODEF}\\
    X(\alpha,t)|_{t=0}&=\alpha.\notag
\end{align}
As a result, we have an ODE for $X$. 

\subsection{Local-in-time existence of solutions}
A standard approach to proving local existence and uniqueness for an ODE is to apply Picard-Lindelöf’s theorem. 
\begin{theorem}[Picard-Lindelöf]\label{th:Picard}
    Let $O\subseteq B$ be an open subset of a Banach space $B$ and let $F:O\to B$ be a locally Lipschitz continuous map. Then, given $X_0\in O,$ there exists a time $T>0$ such that the ordinary differential equation $$\frac{d X}{d t}=F(X), \quad X(\cdot,t=0)=X_0,$$ has a unique (local) solution $X\in C^1((-T,T); O).$
\end{theorem}

{We will apply the above theorem to the operator $F$ defined in \eqref{eq:ODEF}, 
for $B=\LipUa(\R^2; \R^2)$, and, as in the Euler setting, consider
\begin{equation}\label{def:obert}
    O_M = \LipUa(\R^2; \R^2)\cap 
    \left\{ \inf_{\alpha \in \R^2} \det \nabla_\alpha X(\alpha)>\frac12 
    \text{ and } \NUa{X}<M \right\}.
\end{equation}
For any $M>0$, this set is non-empty, open, and consists of one-to-one mappings of $\R^2$ onto $\R^2$. Therefore, it is a valid domain on which one can apply Picard–Lindelöf's theorem.}

We now need to verify that the hypotheses of the theorem hold, that is that $F:O_M\to B$ defined by  $$F(X(\alpha,t))=\int_{\R^2} K(X(\alpha,t)-X(\alpha',t))q_0(\alpha')d\alpha'$$ is bounded and locally Lipschitz on $O_M$. To do so, we first look at the derivatives of the kernel in the distributional sense.

%
%

{We consider each component of the kernel} (denoted with a superscript to avoid confusion with the Bessel functions $K_n$), and using Green’s first identity, we obtain for all $\varphi\in C^\infty_0$ 
\begin{align*}
    \langle \partial_{x_i} K^j, \varphi \rangle &= -\langle  K^j, \partial_{x_i} \varphi \rangle= {-}\lim_{\delta \to 0 }   \int_{|x|\geq \delta} K^j\partial_{x_i} \varphi dx \\ &=\lim_{\delta \to 0 }  \left(\int_{|x|\geq \delta} \partial_{x_i} K^j \varphi dx - \int_{|x|= \delta} K^j \varphi \frac{x_i}{|x|}ds\right)
\end{align*}

By changing variables $x \to \delta x$ on the second integral we get 
\begin{align*}
  \lim_{\delta \to 0 } \int_{|x|= \delta} K^j \varphi \frac{x_i}{|x|}ds&= \lim_{\delta \to 0 } \int_{|x|= 1} K^j(\delta x)\varphi(\delta x) x_i \delta\; ds\\ &
  = \lim_{\delta \to 0 } \int_{|x|= 1} {\frac{\varepsilon}{2\pi}} \frac{{(x^{\perp})_j}}{|x|} K_1(\varepsilon\delta |x|) \varphi(\delta x) x_i \delta\; ds\\
  &=\lim_{\delta \to 0 }  \frac{\varepsilon}{2\pi} \delta K_1(\varepsilon\delta) \int_{|x|= 1}  (x^{\perp})_j \varphi(\delta x) x_i \; ds\\
  &= {\frac{1}{2\pi}}\varphi(0) \int_{|x|=1} (x^{\perp})_j x_i ds 
\end{align*}
where we made use of the dominated convergence theorem and property \eqref{propietat:besselKprop}. 

Therefore, in the sense of distributions $$\partial_{x_i}K^j = p.v. \int_{\R^2} \partial_{x_i}K^j+c_{i,j}\delta_0,\quad c_{i,j}= {-\frac{1}{2\pi}}\int_{|x|= 1} (x^{\perp})_j x_i ds,$$  where $\delta_0$ denotes the Dirac delta at 0,  as usual. For convenience, we adopt the following more compact notation
\begin{equation}\label{eq:derDistK}
    \nabla K = p.v.\int_{\R^2} \nabla K + c \delta_0.
\end{equation}

{To show that $F$ is bounded from $O_M$ to $B$, we proceed in the same spirit as in the Euler setting. 
The estimates required for this step are simpler than those needed later to prove that $F$ is locally 
Lipschitz, and in fact will appear there in a slightly stronger form. For this reason, we do not 
present a separate verification of the boundedness here and simply note that it follows from the 
same type of arguments developed below.
}

Next, to prove that $F$ is locally Lipschitz on $O_M$ we will show a sufficient condition. If the derivative $F'(X)$ is bounded as a linear operator from $O_M$ to $B,$ then the mean-value theorem implies that 
\begin{align*}
    \NUa{F(X_1)-F(X_2)}&=\NUa{\int_0^1\frac{d}{d\delta}F[X_1+\delta(X_2-X_1)]d\delta}\\&\leq \int_0^1\NUa{F'[X_1+\delta(X_2-X_1)]}{d\delta}\NUa{X_1-X_2}
\end{align*}
so $F$ is locally Lipschitz.

Let $X\in O_M$ and $Y\in B$. Then, 
\begin{align*}
    F'(X)Y&=\frac{d}{d\delta}F(X+\delta Y)|_{\delta=0}\\&= \left.\frac{d}{d\delta}\int_{\R^2}K[X(\alpha,t)-X(\alpha',t)+\delta(Y(\alpha,t)-Y(\alpha',t))]q_0(\alpha')d\alpha'\right|_{\delta=0}\\ &=\int_{\R^2} \nabla K(X(\alpha,t)-X(\alpha',t)) (Y(\alpha,t)-Y(\alpha',t))q_0(\alpha')d\alpha',
\end{align*}
where in the last equality the principal value is not needed, since the singularity of $\nabla K$ when $\alpha=\alpha'$ is compensated with the term $Y(\alpha,t)-Y(\alpha',t).$

Thus, if we can prove the following lemma, we have that $F$ satisfies all the hypotheses of Picard-Lindelöf's theorem.
\begin{lemma}\label{lemma:G(x)Y}
    Let $q_0\in\Lipa_c(\R^2;\R)$ for some $0<\gamma<1.$ Given $X\in O_M$ and $Y\in B$, define $$F'(X)Y=\int_{\R^2} \nabla K(X(\alpha,t)-X(\alpha',t)) (Y(\alpha,t)-Y(\alpha',t))q_0(\alpha')d\alpha'.$$ Then there exists $c>0,$ independent of {$\varepsilon,$} $X$ and $Y,$ such that $$\NUa{F'(X)Y}\leq c \Na{q_0}\NUa{Y}.$$  
\end{lemma}

Since the proof of the lemma is rather technical, we postpone it to the end of this subsection.

Combining the result of the lemma with the fact that $F$ is bounded {from $O_M$ to $B$} and Picard-Lindelöf theorem, we obtain:
\begin{theorem}\label{th:EUXlocal}
    Let $q_0\in \Lipa_c(\R^2;\R).$ Then there exists $T(M)>0$ such that the ODE $$\frac{d X}{d t}=F(X), \quad X(\cdot,t=0)=X_0,$$ has a unique solution $X(\cdot,t)\in \LipUa(\R^2;\R^2)$ for $t\in (-T(M),T(M)).$  
\end{theorem}

Moreover, since the velocity field is sufficiently smooth, we also obtain well-posedness in the Hölder class for the \eqref{def:QGSW}.
\begin{theorem}
    For $0<\gamma<1,$ if $q_0\in \Lipa_c(\R^2,\R),$ then there exists $T(M)>0$ such that the \eqref{def:QGSW} has a unique solution $q(\cdot,t)\in \Lipa_c(\R^2,\R)$ for any time $t\in (-T(M),T(M)).$  
\end{theorem}
\begin{proof}
    Assume that $\Tilde{\rho}\in\Lipa_c$ and $\Tilde{v}\in\LipUa$ is also a solution to \eqref{def:QGSW} with the same initial data. Then we can find a trajectory $\Tilde{X}(\cdot,t)$ associated with $\Tilde{v}(\cdot,t)$ such that $\Tilde{\rho}$ is transported by $\Tilde{X}.$ But by uniqueness of the flow map, Theorem \ref{th:EUXlocal}, $X=\Tilde{X}$ so $\Tilde{\rho}(\cdot,t)=\rho_0(\Tilde{X}^{-1}(\cdot,t))=\rho_0(X^{-1}(\cdot,t))=\rho(\cdot,t)$ and by convolving with the kernel we also have that $\Tilde{v}=v.$

    {Finally, we note that all solutions arise from trajectories, since $v\in\LipUa$; see Subsection~\ref{subsec:unqWS} for the general argument in lower regularity. This concludes the proof.}
\end{proof}

We conclude this subsection with the proof of Lemma~\ref{lemma:G(x)Y}.
\begin{proof}[Proof of Lemma \ref{lemma:G(x)Y}]
    In order to ease the notation, we omit the time dependence of $X$ and $Y$. 

    We first estimate $\Np{F'(X)Y}$. Since the set $O_M$ consists of one-to-one and onto functions, the change of variables $X^{-1}(x)=\alpha$ gives 
    \begin{equation*}
        F'(X)Y\circ X^{-1}(x) =\int_{\R^2} \nabla K(x-x') (Y(X^{-1}(x))-Y(X^{-1}(x')))q_0(X^{-1}(x'))dx',
    \end{equation*}
    By the mean value inequality, $$|Y(X^{-1}(x))-Y(X^{-1}(x'))|\leq \Np{(\nabla_\alpha Y\circ X^{-1})\nabla_x X^{-1 }}|x-x'|,$$ and combining this with \eqref{eq:DecGradK} we have that 
    \begin{equation}\label{eq:LipEstimacioInf}
          \begin{aligned}
        |F'(X)Y\circ X^{-1}(x)|&\leq c\Np{(\nabla_\alpha Y\circ X^{-1})\nabla_x X^{-1 }}\Np{f} \\&\left(\int_{|x-x'|\leq 1}|x-x'|^{-2+1}dx'+\int_{\substack{|x-x'|\geq 1\\ x'\in \operatorname{supp}\; f }}|x-x'|^{-2+1}dx'\right)\\
        &\leq  c\Np{(\nabla_\alpha Y\circ X^{-1})\nabla_x X^{-1 }}\Na{f} 
    \end{aligned}
    \end{equation}  
    where, $f(x')\coloneqq q_0(X^{-1}(x'))$, {and the constant $c$ depends on the support of $q_0$.}

    Using \eqref{eq:NaCompInversa} and \eqref{eq:NaGradNabla}, we obtain
    \begin{align*}
        \Na{f(x')} &\leq \Na{q_0\circ X^{-1} (x')}\leq c \Na{q_0} (1+\NUa{X}^{3\gamma}).
    \end{align*}
    {Since $X \in O_M$ (see \eqref{def:obert}), both $\NUa{X}$ and $\Na{\nabla X}$ are uniformly controlled, and thus all $X$–dependent factors are absorbed into the constant $c$.} Applying this to \eqref{eq:LipEstimacioInf} yields
    $$\Np{F'(X)Y}\leq c \NUa{Y} \Na{{q_0}}.$$

    Next, we estimate the Hölder norm of $\nabla F'(X) Y.$ Thus, we need first to calculate the distributional derivative of $H(x)=F'(X)Y\circ X^{-1}(x).$ To do so, we follow the scheme in \cite[p. 165]{MajdaBertozzi} for the Euler equation.   

    Rewrite $H(x)$ as $$H(x)=\int_{\R^2}R(x,x')f(x')dx'$$ $$R(x,x')=\nabla K(x-x') (Y(X^{-1}(x))-Y(X^{-1}(x'))).$$

    We want to compute the distributional derivative of the kernel $R(x+x',x').$ To do so, note previously that, given $h>0$ and $a\in\R^2, |a|=1,$ we have, by Taylor expansion of $Y\circ X^{-1}$ and \eqref{propietat:besselKprop} 
    {$$\lim_{h\to 0} R(x'+ah,x')h=\nabla[Y\circ X^{-1}](x) c(a).$$ }

    {Indeed, recall that
\[
\nabla K \cdot Y =
\begin{pmatrix} \partial_1 K^{(1)} & \partial_2 K^{(1)} \\[1mm] \partial_1 K^{(2)} & \partial_2 K^{(2)} \end{pmatrix} 
\begin{pmatrix} Y_1 \\ Y_2 \end{pmatrix},
\] 
and focus on the contribution from the second row, first column, i.e., $\partial_1 K^{(2)}\, Y_1$; the remaining terms are analogous. Then, from \eqref{eq:KernelPartial1}, we have}

    \begin{align*}
        &\lim_{h\to 0} \partial_1 K^{(2)}(ah)\big(Y_1(X^{-1}(x'+ah))-Y_1(X^{-1}(x'))\big)h \\&= \lim_{h\to 0} \frac{\varepsilon}{2\pi}\left[\frac{a_2^2-a_1^2}{h}K_1(\varepsilon|ah|)-\varepsilon a_1^2 K_0(\varepsilon|ah|)\right] (\nabla[Y_1\circ X^{-1}](x')\cdot ah + o(h)) h
        \\ &= C_a \nabla[Y_1\circ X^{-1}](x'){\cdot a}
        \end{align*}
    where in the last step we use that by \eqref{propietat:besselKprop} we have that $K_1(\varepsilon h ) \sim \frac{1}{\varepsilon h}$ when $h$ is small and $K_0(\varepsilon h)\sim \ln{(\varepsilon h)}$, and therefore $$\lim_{h\to 0} \frac{\varepsilon}{2\pi} \frac{a_2^2-a_1^2}{h}K_1(\varepsilon |ah|) h^2= \frac{a_2^2-a_1^2}{2\pi} C,$$ $$\lim_{h\to 0} \frac{\varepsilon^2}{2\pi} a_1^2 K_0(\varepsilon |ah|)h^2=\lim_{h\to 0}\frac{\varepsilon^2}{2\pi} {a_1^2} \ln{(\varepsilon |ah|)} h^2=0. $$
    
    We are now ready to compute the desired distributional derivative. Let $\varphi\in C^\infty_c(\R^2;\R)$ a test function, then
    
    \begin{align*}
        \langle R(x+x',x'), \partial_{x_k} \varphi \rangle &= {-} \int_{\R^2} \partial_{x_k} \varphi(x) R(x+x',x')dx=-\lim_{\delta\to 0}\int_{|x|\geq \delta} \partial_{x_k} \varphi(x) R(x+x',x')dx\\
        &= \lim_{\delta\to 0}\int_{|x|\geq \delta} \varphi(x) \partial_{x_k} R(x+x',x')dx - \lim_{\delta\to 0}\int_{|x|=\delta} \varphi(x) R(x+x',x') \frac{x_k}{|x|}{ds} \\
        &= p.v. \int \varphi(x) \partial_{x_k} R(x+x',x')dx-c\nabla_\alpha Y(X^{-1}(x))\nabla_xX^{-1}(x)  \varphi(0),
    \end{align*}
    where the equality in the third line follows from Green’s first identity and in the last equality, we used the observation made above {by passing to polar coordinates in the boundary integral, with $x=\delta\,e^{i\theta}$ playing the role of $a h$.}

    {Then, since $\partial_{x_k} R(x, x') = \partial_{x_k} [R(\cdot + x', x')](x - x')$, we get} that the distributional derivative of $H(x)$ can be written as 
    \begin{align*}
        \partial_{x_k}H(x)&=p.v. \int \partial_{x_k} \nabla K(x-x')(Y(X^{-1}(x))-Y(X^{-1}(x')))f(x')dx'\\
        &+p.v. \int \nabla K(x-x')\nabla_\alpha Y(X^{-1}(x)) \partial_{x_k} X^{-1}(x)f(x')dx'\\
        &-c\nabla_\alpha Y(X^{-1}(x))\nabla_xX^{-1}(x)f(x)\eqqcolon I+II+III
    \end{align*}
    Applying \eqref{eq:algebra} again, and estimating $\Na{f}$ as before, it is easy to check that $$\Na{III}\leq c \Na{q_0} \NUa{Y}.$$ 

   {For $II$, we set 
\[
\tilde f(x') \coloneqq \nabla_\alpha Y(X^{-1}(x)) \, \partial_{x_k} X^{-1}(x) \, f(x').
\]
Since $f\in \Lipa_c$ and $X,Y\in O_M$, $\tilde f$ is also in $\Lipa_c$. 
Applying \eqref{eq:SIOepsilon} and \eqref{eq:SIOsna} from Lemma~\ref{lemma:SIO} to the kernel $P = \nabla K$ with $\tilde f$ and using once more the estimate for $\Na{f}$ we obtain
\[
\Na{II}\leq c \Na{q_0} \NUa{Y}.
\]
}


    To estimate $I$ we follow the proof of \cite[Lemma 15]{cantero2022cgamma}. The main difference being that since $K$ or its derivatives are not homogeneous, but one has explicit formulas for them, {so} more computations are needed.
    
    Notice that $Y\circ X^{-1} \in \LipUa.$ Thus, we can write its Taylor series centered at $x'$ as $$Y(X^{-1}(x))=Y(X^{-1}(x'))+\sum_{i=1}^2\partial_i (Y\circ X^{-1})(x')(x_i-x_i')+{\mathcal{R}}(x,x')$$ with $|\mathcal{R}(x,x')|\leq c \NUa{Y\circ X^{-1}} |x-x'|^{1+\gamma}.$ Now, if we add and subtract some terms to $I$ we obtain   
    \begin{align*}
        I & = p.v. \int \partial_{x_k} \nabla K(x-x')(Y(X^{-1}(x))-Y(X^{-1}(x'))-\nabla Y(X^{-1}(x'))(x-x') )f(x')dx' \\ &+ \sum_{i=1}^2 p.v. \int (x_i-x_i')\partial_{x_k} \nabla K(x-x')\partial_i (Y\circ X^{-1})(x')f(x')dx' \coloneqq I_1f(x) + I_2f(x)
    \end{align*}
    
    We compute now some partials of $\nabla K$, the ones not computed are symmetrical to one of the following. To do so, recall the  expression of $\nabla K$ shown on \eqref{eq:KernelPartial1}. 
    \begin{multline*}
        \partial_1 \left(\frac{\varepsilon}{2\pi}\frac{\varepsilon x_1 x_2}{|x|^3}K_2(\varepsilon|x|)\right)= \frac{\varepsilon}{2\pi} \left[ \frac{\varepsilon x_2(x_2^2-x_1^2)}{|x|^4}K_2(\varepsilon|x|)+\frac{\varepsilon^2 x_1^2x_2}{|x|^3}K_2'(\varepsilon|x|)\right] \\=\frac{\varepsilon}{2\pi}\left[ \frac{\varepsilon x_2(x_2^2-x_1^2)}{|x|^4}K_2(\varepsilon|x|)+\frac{\varepsilon^2 x_1^2x_2}{|x|^3}\left(-K_3(\varepsilon |x|) +\frac{2}{\varepsilon |x|}K_2(\varepsilon |x|) \right)\right]
    \end{multline*}
    where in the last equality we used \eqref{propietat:besselDerivades}.

    Looking at the above expression, it is clear that $x_1\partial_1 \left(\frac{\varepsilon}{2\pi}\frac{\varepsilon x_1 x_2}{|x|^3}K_2(\varepsilon|x|)\right)$ has zero integral on spheres. Also, by \eqref{eq:controlBessel} one gets that it is controlled by $\frac{C}{|x|^2}$ with $C$  independent of $\varepsilon$.

    If instead of using the expression used before for $K_2'$ one uses the first equality of \eqref{propietat:besselDerivades}  
    \begin{multline*}
        x_2\partial_1 \left(\frac{\varepsilon}{2\pi}\frac{\varepsilon x_1 x_2}{|x|^3}K_2(\varepsilon|x|)\right) =\\ \frac{\varepsilon^2}{2\pi}\left[\frac{ x_2^2(x_2^2-x_1^2)}{|x|^4}K_2(\varepsilon|x|)+\frac{\varepsilon x_1^2x_2^2}{|x|^3}\left(-K_1(\varepsilon |x|) -\frac{2}{\varepsilon |x|}K_2(\varepsilon |x|) \right) \right]= \\ \frac{\varepsilon^2}{2\pi}\left[
        -\frac{\varepsilon x_1^2 x_2^2}{|x|^3}K_1(\varepsilon |x|)+\frac{x_2^2(x_2^2-3x_1^2)}{|x|^4}K_2(\varepsilon |x|)\right]
    \end{multline*}
    which has a first part that has not zero integral on spheres but it decays faster near the origin, in fact by \eqref{eq:IntdefBessel}, it has finite integral, and a second term that behaves as $x_1\partial_1 \left(\frac{\varepsilon}{2\pi}\frac{\varepsilon x_1 x_2}{|x|^3}K_2(\varepsilon|x|)\right)$. {Indeed, by \eqref{eq:controlBessel} it is clear that both have similar decay and \begin{align*}
\int_{|x|=R} \frac{x_2^2(x_2^2-3x_1^2)}{|x|^4}\,K_2(\varepsilon|x|)\,dS
&= K_2(\varepsilon R)\int_0^{2\pi} 
\bigl(\sin^4\theta - 3\sin^2\theta\cos^2\theta\bigr)\,R\,d\theta \\
&= K_2(\varepsilon R)\,R
\int_0^{2\pi} \frac{ -\cos(2\theta)+\cos(4\theta)}{2}\,d\theta = 0.
\end{align*}
}

    For the rest of the components, we differentiate using the form given in Remark \ref{remark:1derivades} and the first equality in \eqref{propietat:besselDerivades}:
    \begin{multline*}
        \frac{\varepsilon}{2\pi}\partial_1 \left(\frac{x_2^2-x_1^2}{|x|^3}K_1(\varepsilon |x|)-\frac{\varepsilon x_1^2}{|x|^2}K_0(\varepsilon |x|)\right)=\frac{\varepsilon}{2\pi}\Bigg[  \frac{2x_1(x_1^2-3x_ 2^2)}{|x|^5}K_1(\varepsilon |x|)\\-\varepsilon\left(\frac{x_1(x_2^2-x_1^2)}{|x|^4}+\frac{2x_1x_2^2}{|x|^4}\right)K_0(\varepsilon |x|) +\frac{\varepsilon^2 x_1^3}{|x|^3}K_1(\varepsilon |x|)\Bigg]
    \end{multline*}
    it is clear from this expression and \eqref{eq:controlBessel} that when either multiplied by $x_1$ or $x_2$ there is a first term with a decay such as $|x|^{-2}$ with zero integral on spheres and other terms which may or may not have zero integral but are integrable near the origin.

    Finally, by proceeding in the same way,
    \begin{multline*}
        \frac{\varepsilon}{2\pi}\partial_2 \left(\frac{x_2^2-x_1^2}{|x|^3}K_1(\varepsilon |x|)-\frac{\varepsilon x_1^2}{|x|^2}K_0(\varepsilon |x|)\right)= \frac{\varepsilon}{2\pi} \Bigg[ \frac{-2x_2(x_2^2-3x_ 1^2)}{|x|^5}K_1(\varepsilon |x|)\\-\varepsilon\left(\frac{x_2(x_2^2-x_1^2)}{|x|^4}-\frac{2x_1^2x_2}{|x|^4}\right)K_0(\varepsilon |x|) +\frac{\varepsilon^2 x_1^2x_2}{|x|^3}K_1(\varepsilon |x|)\Bigg]
    \end{multline*}
    so the same comment as before can be done. 
    
    Notice that all the integrable terms once written in polar coordinates are like either $\varepsilon^2 rK_0(\varepsilon r)$ or $\varepsilon^3 r^2 K_1(\varepsilon r),$ so after the change of variables $r'=\varepsilon r,$ no $\varepsilon$ remain. Therefore, the constants arising from these terms are independent of $\varepsilon.$   
    
    {From the computations above, it is clear that the kernels $$(x_i-x_i')\partial_{x_k}\nabla K(x-x')$$ behave like $\nabla K$. Hence, we can apply \eqref{eq:SIOsna} to bound $I_2 f(x)$, which can be interpreted as a convolution with these kernels:
\[
\Np{I_2 f} \leq C  \NUa{Y\circ X^{-1}} \Np{f}.
\]
}

    To bound $\Na{I_1f}$ notice that by using the same argument as in the proof of Lemma \ref{lemma:SIO}, $|\partial_{x_k} \nabla K(x-x')|\leq C|x|^{-3},$ and $|\partial_{x_k}\partial_{x_j} \nabla K(x-x')|\leq C|x|^{-4}$. Thus, if we define $$G(x,x')\coloneqq \partial_{x_k} \nabla K(x-x')(Y(X^{-1}(x))-Y(X^{-1}(x'))-\nabla Y(X^{-1}(x'))(x-x') )$$ we  have 
    $$|G(x,x')|\leq \frac{C \NUa{Y\circ X^{-1}}}{|x-x'|^{2-\gamma}} \text{ and } |\nabla_x G(x-x')|\leq \frac{C \NUa{Y\circ X^{-1}}}{|x-x'|^{3-\gamma}}.$$

    Let $x_1,x_2\in \R^2$ and $B\coloneqq B(x_1,3|x_1-x_2|).$ Then
    \begin{align*}
        I_1f(x_1)-I_1f(x_2) &= \int_{\R^2\setminus B} (G(x_1,x')-G(x_2,x'))f(x')dx' \\ &+\int_B G(x_1,x')f(x')dx'-\int_B G(x_2,x')f(x')dx'
    \end{align*}
    so, by the mean value theorem and the bounds for $G$ and $\nabla_x G,$
    $$|I_1f(x_1)-I_1f(x_2)|\leq c \NUa{Y\circ X^{-1}} \Np{f} |x_1-x_2|^\gamma.$$

    Gathering all of the above estimates and bounding $\Na{f}$ as shown, one gets the desired result $$\NUa{F'(X)Y}\leq c \Na{q_0}\NUa{Y}.$$ 
\end{proof}

\subsection{Global existence of smooth solutions}
 In the previous subsection, we showed that the quasi-geostrophic shallow water equation, \eqref{def:QGSW}, admits a unique local-in-time solution if the initial data is in $\Lipa.$ In this section, we claim that such a solution actually exists for all time.

We will only provide a sketch of the proof and omit most of the details, since it closely follows the argument for the Euler equations given in \cite[Chapter 4]{MajdaBertozzi}, and it relies on the following classical continuation result for ODEs.
\begin{theorem}
    Let $O\subset B$ be an open subset of a Banach spaces $B,$ and let $F:O\to B$ be a locally Lipschitz continuous operator. Then the unique solution $X\in C^1([0,T];O)$ to the autonomous ODE, $$\frac{dX}{dt}=F(X),\quad X|_{t=0}=X_0\in O,$$ either exists globally in time or $T<\infty$ and $X(t)$ leaves the open set $O$ as $t\to T.$
\end{theorem}

Since the particle trajectories are volume preserving, as the {velocity field} is divergence-free, the only way that a solution can leave the set $O_M$ is for $\NUa{X}$ to become unbounded as $t\to T^*.$ Thus, we obtain a sufficient condition for global-in-time existence if we give a sufficient condition for $\NUa{X}$ to be a priori bounded. {This, in turn, is guaranteed by an a priori bound on $\int_0^t \Np{\nabla v(\cdot, s)}\,ds$.}


Finally, one can see that $\int_0^t \Np{\nabla v(\cdot, s)} ds$ is controlled by $\int_0^t \Np{ q(\cdot, s)} ds,$ and since $q$ is conserved along particle trajectories it cannot become unbounded. As a consequence, we have proven Theorem \ref{th:EUholder}. That is for $0<\gamma<1,$ if $q_0\in \Lipa_c(\R^2,\R),$ then the \eqref{def:QGSW} has a unique solution $q(\cdot,t)\in \Lipa_c(\R^2,\R)$ for any time $t\in\R$. 

\begin{remark}
    The proofs of this part rely mainly on Grönwall's lemma, and the observation that by \eqref{eq:derDistK} and $v=K*q$ we have $$\nabla v = p.v. \int_{\R^2} \nabla K(x-x')q(x',t)dx'+cq(x,t)$$ and \eqref{eq:SIOepsilon} from Lemma \ref{lemma:SIO}. 

   It is worth noting that, with the previous ingredients, the proof in this case is simpler than that in \cite{MajdaBertozzi}, since they work with the vorticity formulation of the three-dimensional Euler equations, which is not a transport equation.
\end{remark}

\section{Weak solutions to the QGSW equation} \label{sec:weak}
The goal of this section is to address the Yudovich problem, namely, to prove well-posedness for initial data $q_0\in L_c^\infty(\R^2;\R).$ For such non-smooth data, one needs to work with the weak formulation of \eqref{def:QGSW}

\begin{definition}\label{def:weakSol}
    Given $q_0\in L_c^\infty(\R^2;\R),$ the pair $(v,q)$ is a weak solution to \eqref{def:QGSW} with initial data $q_0(x),$ provided that \begin{enumerate}[(i)]
        \item $q\in L^\infty \big([0,T];L_c^\infty(\R^2;\R)\big),$
        \item $v= K*q,$
        \item for all $\varphi\in C^1\big([0,T];C_0^1(\R^2;\R)\big)$ $$\int_{\R^2} \varphi(x,T)q(x,T)dx-\int_{\R^2}\varphi(x,0)q_0(x)dx=\int_0^T\int_{\R^2}\left(\frac{d\varphi}{dt}+v\cdot\nabla\varphi\right) q dxdt$$
    \end{enumerate}
\end{definition}

\subsection{Existence of weak solutions}
{To prove the existence of weak solutions to \eqref{def:QGSW}, we follow the two-step argument presented in \cite[Chapter 8]{MajdaBertozzi}. In this subsection we only give a general idea, since the proof for the Euler equation can be mostly replicated due to the similar properties of the defining kernels, even though the QGSW kernel is not homogeneous. Any significant differences are explicitly mentioned. We recommend consulting \cite[Chapter 8]{MajdaBertozzi} for full details.}

The first step is to smooth the initial data via a convolution with a mollifier and then apply Theorem \ref{th:EUholder} to get pairs $(v^\delta,q^\delta)$ which are smooth solutions, where $\delta$ is the mollification parameter. 

The second step is to extract subsequences $(v^{\delta'})$ and $(q^{\delta'})$ that converge to a pair $(v,q)$ satisfying the weak formulation of Definition \ref{def:weakSol}. This can be formalized in the following proposition:

\begin{prop}
    Let $q_0\in L_c^\infty(\R^2;\R),$ and let $q^\delta, v^\delta$ be a smooth solution on $[0,T]$ with the regularized initial data $$q^\delta_0(x)=\delta^{-2}\int_{\R^2} \rho\left(\frac{x-y}{\delta}\right)q_0(y)dy,$$ for $\delta>0$ and $\rho\in C^\infty_0(\R^2)$ nonnegative and with integral equal to one. Then, for all $t\in[0,T],$ 
    \begin{enumerate}[(i)]
        \item $q^\delta,v^\delta$ are uniformly bounded and $$\Np{v^\delta(\cdot, t)}\leq c\|q^\delta(\cdot,t)\|_{L^1 \cap L^\infty} \leq C\|q_0\|_{L^1 \cap L^\infty},$$
        \item there exist functions $q(\cdot,t)\in L_c^\infty(\R^2;\R)$ and $v=K*q$ such that for all $t\in [0,T]$ $$q^\delta(\cdot,t)\to q(\cdot,t) \text{ in } L^1,$$ $$v^\delta(\cdot,t)\to_x v(\cdot,t) \text{ locally.}$$        
    \end{enumerate}
\end{prop}

{To prove the second part of the previous proposition, the divergence-free property of the velocity field, which ensures volume-preserving trajectories, is crucial. The following lemma, which will be important for the uniqueness, is also needed.}

\begin{lemma}\label{lemma:vlog-lip}
    Let the initial vorticity $q_0\in L_c^\infty(\R^2;\R)$ and let $q^\delta, v^\delta$ be a smooth solution on $[0,T]$ with the regularized initial data $q^\delta_0(x).$ Then $v^\delta(\cdot,t)$ is Log-Lipschitz continuous: $$\sup_{0\leq t \leq T} |v^\delta(x^1,t)-v^\delta(x^2,t)|\leq c \|q_0\|_{L^1\cap L^\infty}\ell(|x^1-x^2|),$$ where $\ell (x)=|x|(1-\ln^-{|x|})$ and $\ln^-{a}=\ln{a}$ for $0<a<1$ and $\ln^-{a}=0$ for $a\geq 1.$
\end{lemma}

The proof of this lemma, follows the standard argument for 2D Euler kernel and can be easily adapted to \eqref{def:QGSW} using the properties in Lemma \ref{lemma:propietatsK}.

We summarize the conclusions of the previous results into the following existence result. 
\begin{theorem}
    Let $q_0\in L_c^\infty(\R^2;\R).$ Then for all time there exists a weak solution $(v,q)$ to \eqref{def:QGSW} in the sense of Definition \ref{def:weakSol}.
\end{theorem}

\subsection{Uniqueness of weak solutions} \label{subsec:unqWS}
For this part we follow \cite[Section 3]{ZinebHaroune}, {where the proof of uniqueness is based on the stability argument for flow maps introduced in \cite{CrippaStefani}}, although the computations are considerably simpler in our setting. We first prove uniqueness of Lagrangian weak solutions, that is, solutions given by the flow map defined in \eqref{eq:ODEF}, and then show that all possible solutions are of this kind, thus proving uniqueness.

Consider $(v,q)$ and $(\Tilde{v}, \Tilde{q})$ two Lagrangian solutions with the same initial data $q_0\in L_c^\infty(\R^2;\R).$ Let $X,\Tilde{X}$ be the unique flows associated to each solution. Then by \eqref{eq:ParticleTrajectories} we have 
\begin{align*}
    X-\Tilde{X}&=\int_0^t[v({X(x,s),s})-\Tilde{v}(\Tilde{X}(x,s),s)]ds \\ 
    &=\int_0^t [v({X(x,s),s})-v(\Tilde{X}(x,s),s)+(v-\Tilde{v})(\Tilde{X}(x,s),s)]ds
\end{align*}

Since $v$ is Log-Lipschitz (see Lemma \ref{lemma:vlog-lip}) we get 
\begin{equation*}
    |X-\Tilde{X}|\leq C \int_0^t \ell (X(x,s)-\Tilde{X}(x,s))ds+\int_0^t|(v-\Tilde{v})(\Tilde{X}(x,s),s)|ds.
\end{equation*}

We focus now on the second term of the previous identity. {For any fixed $s$ we have}
\begin{align*}
    |(v-\Tilde{v})(\Tilde{X}(x,s),s)|&=|(K*q-K*\Tilde{q})(\Tilde{X}(x,s))|\\
    &=\left|\int_{\R^2}K(\Tilde{X}(x,s)-y)q(y)dy-\int_{\R^2}K(\Tilde{X}(x,s)-y)\Tilde{q}(y)dy\right|\\
    &=\left|\int_{\R^2}K(\Tilde{X}(x,s)-X(y,s))q_0(y)dy -\int_{\R^2}K(\Tilde{X}(x,s)-\Tilde{X}(y,s))q_0(y)dy\right|
\end{align*}
where in the last equality we have used that $q$ is transported by the flow.

Putting together the preceding estimates, we obtain 
\begin{multline}\label{eq:difFlowsLag}
      |X-\Tilde{X}|\leq C \int_0^t \ell (X(x,s)-\Tilde{X}(x,s))ds\\+\int_0^t\int_{\R^2}|K(\Tilde{X}(x,s)-X(y,s)) -K(\Tilde{X}(x,s)-\Tilde{X}(y,s))||q_0(y)|dyds
\end{multline}

We now introduce the following Lemma, which is in essence Lemma 2.3 from \cite{ZinebHaroune}, and which will be of great use.  
\begin{lemma}\label{lemma:diffIntKernel}
    For any $f\in L^1(\R^2;\R)\cap L^\infty(\R^2;\R),$ and $K$ the kernel defined in \eqref{def:kernel}, the following holds \begin{enumerate}[(i)]
        \item $\Np{K*f}\leq C \|f\|_{L^1\cap L^\infty}$
        \item For any $x,y\in\R^2$ with $x\neq y$ we have that $$\int_{\R^2}|K(x-z)-K(y-z)||f(z)|dz\leq C \|f\|_{L^1\cap L^\infty} \ell (|x-y|).$$
    \end{enumerate}
\end{lemma}

Let {$\alpha$ be a function in $L^1(\R^2;\R)\cap L^\infty(\R^2;\R)$ with $\alpha>0$}. We also introduce the density $\eta=|q_0|+\alpha$.

We proceed by integrating \eqref{eq:difFlowsLag} against the measure $\eta(x)dx$ and {applying} Fubini's theorem to obtain that 
\begin{multline*}
      \int_{\R^2}|X(x,t)-\Tilde{X}(x,t)|\eta(x)dx\leq C \int_0^t \int_{\R^2}\ell (X(x,s)-\Tilde{X}(x,s))\eta(x)dxds\\+\int_0^t\int_{\R^2}\int_{\R^2}|K(\Tilde{X}(x,s)-X(y,s)) -K(\Tilde{X}(x,s)-\Tilde{X}(y,s))||q_0(y)|dy\eta(x)dxds \\
      =C \int_0^t \int_{\R^2}\ell (X(x,s)-\Tilde{X}(x,s))\eta(x)dxds\\+\int_0^t\int_{\R^2}|q_0(y)|\int_{\R^2}|K(\Tilde{X}(x,s)-X(y,s)) -K(\Tilde{X}(x,s)-\Tilde{X}(y,s))|\eta(x)dxdyds
\end{multline*}

Notice now that the last integral can be bounded using Lemma \ref{lemma:diffIntKernel} (ii), so one gets 
\begin{align*}
      \int_{\R^2}|X(x,t)-\Tilde{X}(x,t)|\eta(x)dx&\leq C \int_0^t \int_{\R^2}\ell (X(x,s)-\Tilde{X}(x,s))\eta(x)dxds\\&+\int_0^t\int_{\R^2}|q_0(y)|\|\eta\|_{L^1\cap L^\infty}\ell (|X(y,s)-\Tilde{X}(y,s)|)dyds\\&\leq C(1+\Np{q}\|\eta\|_{L^1\cap L^\infty}) \int_0^t \int_{\R^2}\ell (X(x,s)-\Tilde{X}(x,s))\eta(x)dxds.
\end{align*}

To conclude, since $\ell$ is concave and $\eta\in L^1,$ we end up with $$\int_{\R^2}|X(x,t)-\Tilde{X}(x,t)|\eta(x)dx\leq C \int_0^t\ell \left(\int_{\R^2} |X(x,t)-\Tilde{X}(x,t)|\eta(x)dx\right)ds,$$ so the uniqueness of the flow-maps follows due to Osgood's Lemma, see, for instance, \cite[Lemma~3.4]{bahouri2011fourier}.

Finally, we prove that all solutions are Lagrangian. {Let $(v,q)$ be a solution with initial data $q_0$.} Define $X$ the flow map associated with the velocity field $v$ as the unique solution of \eqref{eq:ParticleTrajectories}. {Uniqueness of the flow map follows from the log-lipschitz regularity of $v$} (see Lemma \ref{lemma:vlog-lip}). If we now define $Q=q_0(X^{-1}(\cdot,t)),$ it is clear that $Q$ satisfies the same transport equation as $q$ with velocity $v$. Therefore, by classical results on transport equations (see, for instance, \cite{diperna1989ordinary}) we deduce that $Q = q,$ so $q$ is a Lagrangian solution. 

We have thus proved the following result.
\begin{theorem}\label{th:yudovichQGSW}
    Let $q_0\in L_c^\infty(\R^2;\R).$ Then for all time there exists a unique weak solution $(v,q)$ to \eqref{def:QGSW} in the sense of Definition \ref{def:weakSol}.
\end{theorem}

\section{Vortex patches}\label{sec:patch}
    Weak solutions introduced in the previous section have well-defined particle trajectories along which the vorticity is preserved. If we start at some time with a vorticity field that is constant inside some region, the region evolves in time keeping the same constant vorticity inside. This fact leads to the natural class of weak solutions known as vortex patches, previously introduced.

    If the boundary of the patch is sufficiently smooth, then the velocity field induced by the patch can be completely determined by the position of the boundary. If the boundary is at least piecewise $C^1,$ then the velocity induced by the patch at time $t$, $\Omega_t$, is given by   $$v=-\frac{1}{2\pi}\int_{\Omega_t} \nabla^\perp K_0(\varepsilon |x-x'|)dx',$$
    thus applying Green's formula yields
    $$v=\frac{1}{2\pi}\int_{\partial \Omega_t} K_0(\varepsilon |x-x'|) [-n_2(x'),n_1(x')]^tdx'.$$
    
    Therefore, the motion of the boundary of the patch satisfies the contour dynamics equation (CDE):    
	\begin{equation*}\label{def:CDE}\tag{CDE}
		\left\{\begin{aligned}
			&\frac{\partial z}{\partial t}(\alpha,t)=\frac{1}{2\pi}{\int_{S^1}} K_0(\varepsilon |z(\alpha,t)-z(\alpha',t)|)z_\alpha (\alpha',t)d\alpha'\\ 
			&z(\alpha,t)|_{t=0}=z_0(\alpha)
		\end{aligned}\right.
	\end{equation*}

   where $z(\alpha,t)$ is a Lagrangian parametrization of the curve at time $t$, that is $z(\alpha,t)={X(z(\alpha,0),t)},$ where $X$ is the particle trajectory map for the flow.

   {Notice that taking \(\varepsilon \to 0\) in the above expression does not directly yield the Euler contour dynamics equation. However, from the derivation above using Green’s formula, it is clear that subtracting a constant from the kernel does not affect the integral over a closed curve. Hence, one can replace $K_0$ by $$\Tilde{K_0} \coloneqq K_0 - \ln \frac{\varepsilon}{2}$$ in \eqref{def:CDE} to recover the Euler equation in the limit case. }
   
\subsection{Local existence and uniqueness of solutions to the CDE}
    We will apply Picard's Theorem (Theorem \ref{th:Picard}) to the study of solutions to the \eqref{def:CDE}. The Banach space we will consider is $\LipUa(S^1; \R^2)$ and as an open set we define for any $M>0$ the set
	$$O^M= \left\{z(\cdot,t)\in\LipUa \;|\; \frac{1}{M}<\frac{|z(\alpha,t)-z(\alpha',t)|}{|\alpha-\alpha'|}<M\right\}$$
	which is nonempty and consists only of one-to-one mappings of $S^1$ to $\R^2.$
	
	\begin{prop}\label{prop:CDEpicard}
		Let $F:O^M\to \LipUa(S^1; \R^2)$ be defined by 
		$$F(z(\alpha,{t}))\coloneqq \int_{S^1} K_0(\varepsilon |z(\alpha,t)-z(\alpha',t)|)z_\alpha (\alpha',t)d\alpha'.$$
		Then $F$ is a locally Lipschitz continuous operator from $O^M$ to $\LipUa(S^1; \R^2)$.
	\end{prop}

    In order to prove the result we will use the following lemma:
    

    
    {
    \begin{lemma}\label{lemma:nuclisBons}
        Let $K:S^1\times S^1\setminus\{\alpha=\alpha'\}\to\R$ be a kernel such that
        \begin{equation}\label{eq:kernelAssumptions}
        \sup_{\alpha\in S^1}\int_{S^1} |K(\alpha,\alpha')|\,d\alpha' < \infty,
        \qquad
        \sup_{\alpha\in S^1}\int_{S^1} |\partial_\alpha K(\alpha,\alpha')|\,d\alpha' < \infty.
        \end{equation}
        For $f\in L^\infty(S^1)$ define
        \[
        Tf(\alpha)\coloneqq \int_{S^1} K(\alpha,\alpha')f(\alpha')\,d\alpha'.
        \]
        Then $Tf$ is Lipschitz on $S^1$, i.e.  $Tf\in \operatorname{Lip}(S^1)$, and there exists a constant $C>0$, depending only on the bounds in
        \eqref{eq:kernelAssumptions}, such that
        $$\|Tf\|_{L^\infty}+\Na[1]{Tf} \le C \|f\|_{L^\infty},$$ where $\Na[1]{\cdot}$ denotes the Lipschitz norm (that is Definition \ref{def:holder} with $\gamma=1$).
    \end{lemma}}

    \begin{proof}
        By the first assumption in \eqref{eq:kernelAssumptions}, for any $\alpha\in S^1$ we have
        \[
        |Tf(\alpha)| \le \|f\|_{L^\infty}\int_{S^1} |K(\alpha,\alpha')|\,d\alpha'
        \le C\|f\|_{L^\infty},
        \]
        and therefore $\|Tf\|_{L^\infty}\le C\|f\|_{L^\infty}$.
        
        To estimate the Lipschitz seminorm, let $\alpha_1,\alpha_2\in S^1$ and set
        $d\coloneqq |\alpha_1-\alpha_2|$. We decompose
        \begin{equation*}
        \begin{aligned}
        Tf(\alpha_1)-Tf(\alpha_2)
        &=\int_{|\alpha_1-\alpha'|<3d} K(\alpha_1,\alpha')f(\alpha')\,d\alpha' -\int_{|\alpha_1-\alpha'|<3d} K(\alpha_2,\alpha')f(\alpha')\,d\alpha' \\
        &\quad +\int_{|\alpha_1-\alpha'|\ge 3d}
        \big(K(\alpha_1,\alpha')-K(\alpha_2,\alpha')\big)f(\alpha')\,d\alpha' \eqqcolon I_1+I_2+I_3.
        \end{aligned}
        \end{equation*}
        
        Using the boundedness of $K$, we obtain
        \[
        |I_1| \le \|f\|_{L^\infty}\sup_{\alpha,\alpha'}|K(\alpha,\alpha')|
        \int_{|\alpha_1-\alpha'|<3d} d\alpha'
        \le C d\,\|f\|_{L^\infty},
        \]
        and the same bound holds for $I_2$.
        
        For $I_3$, by the Mean Value Inequality there exists $\tilde\alpha$ between $\alpha_1$ and
        $\alpha_2$ such that
        \[
        |K(\alpha_1,\alpha')-K(\alpha_2,\alpha')|
        \le d\,|\partial_\alpha K(\tilde\alpha,\alpha')|.
        \]
        Therefore,
        \[
        |I_3|
        \le d\,\|f\|_{L^\infty}
        \int_{|\alpha_1-\alpha'|\ge 3d} |\partial_\alpha K(\tilde\alpha,\alpha')|\,d\alpha'
        \le C d\,\|f\|_{L^\infty},
        \]
        using the second assumption in \eqref{eq:kernelAssumptions}.
        
        Combining the above estimates yields
        \[
        |Tf(\alpha_1)-Tf(\alpha_2)| \le C d\,\|f\|_{L^\infty},
        \]
        which shows that $Tf\in \operatorname{Lip}(S^1)$ and completes the proof.
    \end{proof}   

    We are now ready to prove that $F$ satisfies Picard's theorem conditions.
    
	\begin{proof}[Proof of Proposition \ref{prop:CDEpicard}] 
		We will use the series expansions of the kernel, $K_0$, and its derivative, $-K_1$, which we rewrite here:
		\begin{align}\label{eq:reDefiKernels}
        \begin{split}    
			K_0(r)=-\log\left(\frac{r}{2}\right)-\log\left(\frac{r}{2}\right)h_1(r)+h_2(r),\\ 
			K_1(r)=\frac{1}{r}+\frac{r}{2}\ln\left(\frac{r}{2}\right)g_1(r)-\frac{r}{4}g_2(r),
            \end{split}
		\end{align}
		where $h_1, h_2, g_1, g_2$ are analytic functions given by   
		$$h_1(r)\coloneqq \sum_{m=1}^\infty \frac{\left(\frac{r}{2}\right)^{2m}}{(m!)^2}, \quad h_2(r)\coloneqq \sum_{m=0}^\infty\frac{\left(\frac{r}{2}\right)^{2m}}{(m!)^2}\psi(m+1),$$
		$$g_1(r)\coloneqq \sum_{m=0}^\infty \frac{\left(\frac{r}{2}\right)^{2m}}{m!(m+1)!},\quad g_2(r)\coloneqq \sum_{k=0}^\infty (\psi (k+1) + \psi(k+2))\frac{\left(\frac{r}{2}\right)^{2k}}{k!(k+1)!},$$    
		being $ \psi(1)=-\gamma, \; \psi(m+1)=\sum_{k=1}^m\frac{1}{k}-\gamma,$ being $\gamma$ Euler's constant.
		
		We see first that $F$ is a bounded operator from $O^M$ to $\LipUa(S^1; \R^2)$. To do so we start by bounding $\Na{F}$ (we do not really need to bound the $\gamma$ seminorm, but the computations will be useful later).
		
		Using the kernel series expression, and omitting the time dependence to ease the notation, we have 
		\begin{equation}\label{eq:decomposicioF}
			\begin{aligned}
				F(z(\alpha))&= \int_0^{2\pi} K_0(\varepsilon |z(\alpha)-z(\alpha')|) z_\alpha (\alpha') d\alpha' \\
				&= -\int_0^{2\pi} \left(\log|z(\alpha)-z(\alpha' )|+\ln{\frac\varepsilon2}\right)z_\alpha (\alpha' ) d\alpha'\\&-
				\int_0^{2\pi}\left(\log|z(\alpha)-z(\alpha' )|+\ln{\frac\varepsilon2}\right)h_1(\varepsilon|z(\alpha )-z(\alpha' )|)z_\alpha (\alpha' ) d\alpha'\\&+ \int_0^{2\pi} h_2(\varepsilon|z(\alpha )-z(\alpha' )|) z_\alpha (\alpha' ) d\alpha'\eqqcolon I_1+ I_2+ I_3
			\end{aligned}        
		\end{equation}
		
		Notice that $I_1$ corresponds to the CDE operator of Euler's equation ({and coincides exactly with it when one considers} $\Tilde{K_0}$ instead of $K_0$). Therefore, by \cite[Lemma 8.10]{MajdaBertozzi} we have that for $z\in O^M,$ $$\Na{I_1}\leq C_{\gamma,M} (\Np{z_\alpha} +1 ), \; \Na{\nabla_\alpha I_1}\leq C(\Na{z}+1).$$
		
		For $I_2$ we have 
		\begin{equation*}
			\begin{aligned}
				I_2&= - \int_0^{2\pi}\log|z(\alpha)-z(\alpha' )|h_1(\varepsilon|z(\alpha )-z(\alpha' )|)z_\alpha (\alpha' ) d\alpha' \\&- \ln{\frac\varepsilon2} \int_0^{2\pi}h_1(\varepsilon|z(\alpha )-z(\alpha' )|)z_\alpha (\alpha' ) d\alpha'\coloneqq I_{21} + I_{22}.
			\end{aligned}
		\end{equation*}   

        {Both $I_{21}$ and $I_{22}$ are of the form $Tf(\alpha)=\int_0^{2\pi} \mathcal{K}(\alpha,\alpha')z_\alpha(\alpha') d\alpha',$ where the kernel satisfies the conditions of Lemma \ref{lemma:nuclisBons}, therefore $\Na{I_{2}}\leq C_\varepsilon \Np{z_\alpha}.$ The term $I_3$ can be treated in the same way with $\mathcal{K}(\alpha,\alpha')=h_2(\varepsilon|z(\alpha)-z(\alpha')|).$ Moreover, one can see that, in both cases, $C_\varepsilon\to 0$ as $\varepsilon\to 0$} 

        Thus, combining the above estimates one gets $$\Na{F}\leq C (\Np{z_\alpha}+1).$$

        We now bound $\Na{\nabla_\alpha F}.$ To do so, we first decompose the functional as in \eqref{eq:decomposicioF} so
        $$\Na{\nabla_\alpha F} \leq \Na{\nabla_\alpha I_1} + \Na{\nabla_\alpha I_2} + \Na{\nabla_\alpha I_3}. $$

        Since $I_1$ corresponds to the Euler part, its contribution has already been handled above. For the other two cases one notices that it is possible to differentiate under the integral sign. Moreover, one can check that all of the terms resulting from differentiating with respect to $\alpha$ satisfy the conditions of Lemma \ref{lemma:nuclisBons}. {Indeed, notice that the worst term comes from differentiating $I_{21},$ but since $h_1(r)\sim r^2$ as $r\to0$, the resulting kernel behaves like $|\alpha-\alpha'|\ln|\alpha-\alpha'|$  which satisfies conditions of Lemma \ref{lemma:nuclisBons}. In particular, the non-Eulerian contributions $I_2$ and $I_3$ enjoy better regularity: they define Lipschitz operators, and hence their Hölder bounds follow directly.}
        
		We have so far seen that $F$ is a bounded operator from $O^M\to \LipUa$. We now see that it is locally Lipschitz. To do so, we bound its Gateaux derivative as in previous sections. For $y\in O^M$
		\begin{equation*}
			\begin{aligned}
				F'(z(\alpha))y&=\frac{d}{d\delta}(F(z(\alpha)+\delta y(\alpha)))\big|_{\delta=0}\\
				&=\int_0^{2\pi} - K_1(\varepsilon |z(\alpha)-z(\alpha')|)\varepsilon \frac{(y(\alpha)-y(\alpha'))(z(\alpha)-z(\alpha'))}{|z(\alpha)-z(\alpha')|}z_\alpha (\alpha')d\alpha'\\
				&+\int_0^{2\pi} K_0(\varepsilon |z(\alpha)-z(\alpha')|) y_\alpha (\alpha') d\alpha' \eqqcolon G_1(z)y+G_2(z)y
			\end{aligned}
		\end{equation*}
		
		Notice that $G_2(z)y$ is essentially like $F(z)$. Therefore, one can repeat the same arguments to see that it is bounded on $\LipUa$. 
		
		We just have to see that $\NUa{G_1(z)}$ is bounded. To do so, we will use the series expansion on $K_1$:
		\begin{equation*}
			\begin{aligned}
			G_1(z)y&= -\int_0^{2\pi}  \frac{(y(\alpha)-y(\alpha'))(z(\alpha)-z(\alpha'))}{|z(\alpha)-z(\alpha')|^2}z_\alpha (\alpha')d\alpha' \\ 
			&- \int_0^{2\pi} g(\varepsilon |z(\alpha)-z(\alpha')|) \frac{(y(\alpha)-y(\alpha'))(z(\alpha)-z(\alpha'))}{|z(\alpha)-z(\alpha')|}z_\alpha (\alpha')d\alpha'\\
    &\eqqcolon G_{11}(z)y- G_{12}(z)y
			\end{aligned}
		\end{equation*}  		
		where $g(r)\coloneqq\frac{1}{2} r\left(\ln{\frac{r}{2}}g_1(r)-\frac{1}{2}g_2(r)\right)$.
		
		Now one can see that $G_{11}(z)y$ corresponds to one of the terms concerning Euler equation, thus following \cite[Proposition 8.9]{MajdaBertozzi} one gets the desired bound.

        It only remains to bound $\NUa{G_{12}(z)y}$. To do so first notice that since $y,z\in O^M$ and $g\in L^\infty([0,2\pi];\R)$ it is clear that $\Np{G_{12}(z)y}<\infty$. Thus, we compute the $\alpha$ derivative:
        
        \begin{equation*}
        \begin{aligned}
            \frac{dG_{12}(z)y}{d\alpha} &= \int_0^{2\pi} \varepsilon g'(\varepsilon |z(\alpha)-z(\alpha')|) \frac{(y(\alpha)-y(\alpha'))(z(\alpha)-z(\alpha'))^2}{|z(\alpha)-z(\alpha')|^2}z_\alpha (\alpha') z_\alpha (\alpha) d\alpha'\\ 
            &+\int_0^{2\pi} g(\varepsilon |z(\alpha)-z(\alpha')|) \frac{y_\alpha(\alpha)(z(\alpha)-z(\alpha'))}{|z(\alpha)-z(\alpha')|}z_\alpha (\alpha') d\alpha'\\
            &+\int_0^{2\pi} g(\varepsilon|z(\alpha)-z(\alpha')|) \frac{ (y(\alpha)-y(\alpha'))z_\alpha(\alpha)}{|z(\alpha)-z(\alpha')|}z_\alpha(\alpha')d\alpha'\\
            &-\int_0^{2\pi} g(\varepsilon |z(\alpha)-z(\alpha')|)\frac{(y(\alpha)-y(\alpha'))(z(\alpha)-z(\alpha'))^2}{|z(\alpha)-z(\alpha')|^3}z_\alpha (\alpha') z_\alpha (\alpha) d\alpha'\\&\eqqcolon J_1+J_2+J_3+J_4
        \end{aligned}
        \end{equation*}

        Notice now that on each term one can apply Lemma \ref{lemma:nuclisBons}, which finishes the proof. Indeed, for instance for $J_4$ one considers the kernel 
        \begin{align*}
            \mathcal{K}_4(\alpha,\alpha')&=g(\varepsilon |z(\alpha)-z(\alpha')|)\frac{(y(\alpha)-y(\alpha'))(z(\alpha)-z(\alpha'))^2}{|z(\alpha)-z(\alpha')|^3},
        \end{align*} acting on $z_\alpha(\alpha').$ {Using that $y,z\in O^M$ and the definition of $g(r)$ it is clear that $|\mathcal{K}_4(\alpha,\alpha')|\leq C |\alpha - \alpha'|\ln{|\alpha - \alpha'|}$ and $|\nabla_ \alpha\mathcal{K}_4(\alpha,\alpha')|\leq C \ln{|\alpha - \alpha'|}$ so the conditions in the Lemma hold. }
	\end{proof}

    \begin{remark}
          Notice that on all of the estimates done in the previous Lemma, if one considers initially $\Tilde{K}_0$ instead of $K_0$, then the constants corresponding to all terms containing $\varepsilon$ tend to 0, with the worst ones behaving like $\mathcal{O}(\varepsilon \ln \varepsilon)$ as $\varepsilon \to 0$. 
    \end{remark}

     \subsection{Global regularity for vortex patches}
    In the previous subsection we proved local in time existence and uniqueness of solution to the \eqref{def:CDE}, and since we are dealing with a transport problem this is equivalent to dealing with the vortex patch problem. Therefore, we have seen that if $\partial\Omega\in\LipUa$, being $\Omega$ a bounded, simply connected region of constant vorticity,  there exists a unique solution to the vortex patch problem that remains regular on a time interval $[0,T^*]$. Our goal now is to prove that the solution remains regular for all time.  To this end, we follow the approach of \cite[Chapter 8]{MajdaBertozzi}. We also note that the alternative approach introduced in \cite{VerderaRevisited}, which avoids the use of defining functions, could likely be adapted to the present setting.
    
    In order to do so, we reformulate the vortex patch problem in terms of a scalar function, $\varphi(x,t)$, that defines the patch boundary by $\{\varphi(x,t)>0\}=\Omega(t)$ and is convected with the flow by 
    \begin{equation}\label{eq:defFun}
        \begin{aligned}
            \frac{\partial \varphi}{\partial t}+v\cdot \nabla \varphi &= 0,\\
            \varphi(x,0)&=\varphi_0(x).
        \end{aligned}
    \end{equation}

    The function $\varphi$ is usually referred to as the \textit{defining function} of the patch, and the previous equation guarantees that it defines the patch boundary at later times if it initially satisfies $\{\varphi(x,0)>0\}=\Omega_0$. By the weak solution theory and the equivalence with the \eqref{def:CDE} for bounded domains with smooth boundaries, one gets that if $X(x,t)$ are the particle trajectories associated with initial data $\chi_{\Omega_0}$, then a solution to the previous equation is just $\varphi(x,t)=\varphi_0(X^{{-1}}(x,t)).$ 


    As in the Euler case, we prove global regularity for the patch boundary by the equivalence of the evolution of $\varphi$ with vortex patch equation. 
    \begin{theorem} \label{th:propDef}
        Given $\Omega_0$ bounded, and $\varphi\in \LipUa(\R^2;\R)$ satisfying $$\{\varphi(x,0)>0\}=\Omega_0 \text{ and } \inf_{x\in \partial\Omega_0} |\nabla \varphi(x,0)|\geq m > 0,$$ there exists a constant $C$ so that \eqref{eq:defFun} has a unique solution $\varphi(x,t)$ defined for all $x\in \R^2$ and all $t\in \R$ and satisfies
        \begin{equation*}
            \begin{aligned}
                \Np{\nabla v(\cdot,t)}&\leq \Np{\nabla v (\cdot,0)}e^{C|t|},\\
                \SNa{\nabla \varphi(\cdot, t)}&\leq \SNa{\nabla \varphi(\cdot, 0)}{\exp}[(C_0+\gamma)e^{C|t|}],\\
                \Np{\nabla \varphi(\cdot, t)}&\leq \Np{\nabla \varphi(\cdot, 0)}\exp[e^{C|t|}],\\
                |\nabla \varphi(\cdot, t)|_{\inf}&\geq |\nabla \varphi(\cdot, 0)|_{\inf}\exp[-e^{C|t|}].
            \end{aligned}
        \end{equation*}
        Here $C_0$ is a fixed constant and $|\nabla \varphi(\cdot)|_{\inf} = \inf_{x\in\partial\Omega} |\nabla \varphi(x)|$
    \end{theorem}

    The proof of Theorem \ref{th:propDef} comes from combining the following results. For the details, see \cite[\S 8.3.3]{MajdaBertozzi}.
    \begin{prop}\label{prop:geometric}
        Assume that $v$ is given by $v=\nabla_x^\perp K_0 * \chi_{\Omega}$ and $\varphi$ such that $\{\varphi(x,0)>0\}=\Omega_0.$ Then $$\Np{\nabla v}\leq C_\gamma \left(1+\ln\frac{\SNa{\nabla\varphi(\cdot)} area(\Omega_0)^\frac{\gamma}{2}}{|\nabla\varphi(\cdot)|_{\inf}}\right).$$
    \end{prop}

    \begin{prop}\label{prop:commutador}
        Assume that $v$ is given by $v=\nabla_x^\perp K_0 * \chi_{\Omega}$ and $W$ is a divergence-free vector field tangent to $\partial \Omega,$ then $$\nabla v(x)W ={\frac{1}{2\pi}} p.v. \int_{\Omega} \nabla_x[\nabla_x^\perp K_0(\varepsilon |x-y|)][W(x)-W(y)]dy.$$
        And, if $W\in\Lipa(\R^2;\R^2)$ then there exists a constant $C_0$ such that 
        \begin{equation}\label{eq:boundcommutador}
            \SNa{\nabla v W}\leq C_0 \Np{\nabla v}\SNa{W}.
        \end{equation}
    \end{prop}

    {\begin{remark}
        The appearance of the commutator is not unexpected: the singular, distributional part of $\nabla\nabla^\perp K_0(\varepsilon|x-y|)$ coincides with the one coming from the Euler kernel (as seen in \eqref{eq:derDistK}). The additional terms in the series expansion of $K_0$ are smoother and do not contribute to the delta-type part of the kernel. Thus the same cancellation mechanism that produces the Euler commutator carries over directly to our case.
    \end{remark}}

    \begin{prop}
        We assume that $\varphi$ is a solution of the evolution problem on some time interval of time $|t|<T.$ We also assume that $\varphi(\cdot,t)\in \LipUa(\R^2;\R)$ and that $|\nabla \varphi(\cdot,t)|_{\inf}>0$ for $|t|\leq T.$ Then
        \begin{equation*}
            \begin{aligned}
                \SNa{\nabla \varphi(\cdot,t)}&\leq \SNa{\nabla \varphi(\cdot,0)}{\exp}\left[(C_0+\gamma)\int_0^t\Np{\nabla v(\cdot,s)}ds\right],\\
                \Np{\nabla \varphi(\cdot,t)}&\leq \Np{\nabla \varphi(\cdot,0)}\exp\left[\int_0^t \Np{\nabla v(\cdot,s)}ds\right],\\
                |{\nabla \varphi(\cdot,t)}|_{\inf}&\geq |{\nabla \varphi(\cdot,0)}|_{\inf} \exp\left[-\int_0^t\Np{\nabla v(\cdot,s)}ds\right].
            \end{aligned}
        \end{equation*}
    \end{prop}

    We will only prove Proposition \ref{prop:commutador}, since it is the one that differs the most from the equivalent propositions for the Euler case. {For the other two, the same arguments done for the Euler case work here. The only adjustment occurs in the proof of Proposition \ref{prop:geometric}, where we use the kernel decomposition $\SU+\SD$ from Remark \ref{remark:1derivades}. The singular part $\SU$ is treated exactly as the Euler kernel, yielding the same logarithmic estimate, while the $L^1$ part $\SD$ contributes an additive constant to the $L^\infty$ bound of $\nabla v$. This constant is handled identically to the antisymmetric constant $c\chi_{\Omega}$ in Euler’s kernel: it does not affect the logarithmic dependence on the boundary smoothness and is absorbed into the coefficients of the Gronwall inequality, leaving the structure of the proof and the conclusion of global regularity unchanged.}

    \begin{proof}[Proof of Proposition \ref{prop:commutador}]
        From \eqref{eq:derDistK} we have that 
        \begin{equation}\label{eq:derDistXi}
            \nabla v(x) = {\frac{1}{2\pi}} p.v.\int_{\Omega} \nabla[\nabla^\perp K_0(\varepsilon|x-y|)]dy  + c \chi_\Omega(x),
        \end{equation}        
        where {$c= \frac{1}{2} \begin{pmatrix} 0 & -1 \\ 1 & 0\end{pmatrix}.$}  
        
        We use Green's formula and the fact that $W$ is divergence free and tangent to $\partial \Omega$ to rewrite the above expression:
        \begin{equation*}
            \begin{aligned}
                \frac{1}{2\pi} &p.v. \int_{\Omega} \nabla[\nabla^\perp K_0(\varepsilon|x-y|)]W(y)dy \\
                &=-\lim_{\delta \to 0} \frac{1}{2\pi}\int_{|x-z|=\delta,\, z\in \Omega}\left[ W(z)\cdot \frac{x-z}{\delta}\right] \nabla^\perp K_0(\varepsilon |x-z|) dz
            \end{aligned}
        \end{equation*}

        By setting $x-z=\delta \theta $ with $|\theta|=1$, and using the series expansion of $K_0(r),$  we can rewrite as 
        \begin{equation*}
            \begin{aligned}
                &-\lim_{\delta \to 0} \frac{1}{2\pi}\int_{|\theta|=1}\left[ W(x-\delta\theta)\cdot \theta \right] \nabla^\perp K_0(\varepsilon \delta) \delta d\sigma(\theta)\\
                &=-\lim_{\delta \to 0} \frac{1}{2\pi}\int_{|\theta|=1}\left[ W(x-\delta\theta)\cdot \theta \right] \left[\frac{-\delta \theta^\perp}{\delta^2} +\mathcal{O}(1)\right]\delta d\sigma(\theta)
            \end{aligned}
        \end{equation*}

        Finally, write $W(x)=(W_1(x),W_2(x))$ and $\theta=(\theta_1,\theta_2)$ to obtain
        \begin{equation*}
        \begin{aligned}
        &\frac{1}{2\pi}\chi_{\Omega}(x)\int_{|\theta|=1} (W(x)\cdot\theta)\,\theta^\perp\, d\sigma(\theta) \\
        &= \frac{1}{2\pi}\chi_{\Omega}(x)\Bigg(
        W_1(x)\int_{|\theta|=1} \theta_1\theta^\perp\, d\sigma(\theta)
        + W_2(x)\int_{|\theta|=1} \theta_2\theta^\perp\, d\sigma(\theta)
        \Bigg)=-cW(x) \chi_{\Omega}(x),
        \end{aligned}
        \end{equation*}

        {where $c$ is the same $2\times2$ constant matrix as in \eqref{eq:derDistXi}. Therefore, we can substitute the second term of \eqref{eq:derDistXi} to get }
         $$\nabla v(x)W = \frac{1}{2\pi} p.v. \int_{\Omega} \nabla_x[\nabla_x^\perp K_0(\varepsilon |x-y|)][W(x)-W(y)]dy.$$

         We now see the second part of the proposition. To do so, consider $$G(x)=p.v.\int_{\R^2} H(x-y)[W(x)-W(y)]\chi_\Omega(y),$$ where $H(x-y)\coloneqq \nabla[ \nabla^\perp K_0({\varepsilon|x-y|})]$. 
         
         Recall from Lemma \ref{lemma:propietatsK} and {Remark \ref{remark:1derivades}} that $|H(x-y)|\leq \frac{C}{|x-y|^2}$ and that $H(x-y)={\SU+\SD}$ where $\SU$ has mean value zero on spheres and $\SD$ is integrable. Moreover, by using \eqref{propietat:besselRecurrencia} one can see that $|\nabla H(x-y)|\leq \frac{C}{|x-y|^3},$ as done in the proof of Lemma \ref{lemma:SIO}.  We use this properties to bound $\SNa{G}$:
        
        \begin{equation*}
        \begin{aligned}
            G(x)-G(x+h) &= p.v. \int H(x-y)[W(x)-W(y)]{\chi_\Omega(y)}dy\\
            &-p.v. \int H(x+h-y)[W(x+h)-W(y)]\chi_\Omega(y)dy\\ 
            &=p.v. \int_{|x-y|<2h} H(x-y)[W(x)-W(y)]\chi_\Omega(y)dy\\&- p.v. \int_{|x-y|<2h} H(x+h-y)[W(x+h)-W(y)]\chi_\Omega(y)dy\\
            &+p.v. \int_{|x-y|\geq 2h} H(x-y)[W(x)-W(x+h)]\chi_\Omega(y)dy\\&- p.v. \int_{|x-y|\geq 2h}[H(x-y)-H(x+h-y)][W(x+h)-W(y)]\chi_\Omega(y)dy\\ &= I_1+ I_2+I_3+I_4
        \end{aligned}            
        \end{equation*}

        By the stated properties on $H$, it is clear that $|I_1|,|I_2|\leq C\SNa{W}h^\gamma \Np{\chi_\Omega}.$ Also $$|I_4|\leq \int_{|x-y|\geq 2h} h \frac{C}{|x-y|^{3-\gamma}}\SNa{W}\Np{\chi_\Omega} \leq C\SNa{W}h^\gamma \Np{\chi_\Omega}.$$

        To bound $|I_3|$ notice that the properties aforementioned on $H$ imply also the following ones: 
         \begin{enumerate}
             \item For each $0<\delta<N,$ $|\int_{\delta<|x|<N}H(x)dx|\leq c_1,$ for $c_1$ a constant independent of ${\varepsilon},\; N,$ and furthermore $\lim_{\delta\to 0}\int_{\delta <|x|<N} H(x)dx$ exists for each fixed $N$. \label{P1}
             \item $\int_{|x|\leq R}|x||H(x)|dx\leq c_2R,$ for all $R>0,$ with $c_2$ independent of $R.$  \label{P2}
             \item $\int_{|x|\geq 2|y|} |H(x-y)-H(x)|dx\leq c_3,$ $y\neq 0,$ and $c_3$ independent of $y$.  \label{P3}
             \item If $x_1, x_2, x_3, y$ are such that $|x_1-x_2|,|x_2-x_3|, |x_1-x_3|\leq R/2$ and $|x_1-y|, |x_2-y|,|x_3-y|\geq R,$ then $$|H(x_1-y)-H(x_2-y)|\leq c_4\frac{|x_1-x_2|}{|x_3-y|^{3}}.$$ \label{P4}
         \end{enumerate}

         Indeed, the decomposition of $H$ into ${\SU,\SD}$ implies \ref{P1}. {Property} \ref{P2} is implied by the growth condition on $H$ and \ref{P3} and \ref{P4} by the growth condition on $\nabla H.$ Properties \ref{P1}-\ref{P4} allow for the application of Cotlar's lemma (see \cite[p. 291]{torchinsky1986real})

         Therefore, we have that $$|I_3|\leq \SNa{W}h^\gamma \left|p.v.\int_{|x-y|\leq 2h} H(x-y)\chi_\Omega(y)dy \right| \leq C \SNa{W} h^\gamma (\Np{H* \chi_\Omega}+\Np{\chi_\Omega}),$$
         where in the last inequality we used Cotlar's Lemma. 

        Finally, to get \eqref{eq:boundcommutador} it remains to show that $\Np{H* \chi_\Omega}$ is bounded. To do so, we use again the decomposition of $H$. For $\SD$ the result is clear and for $\SU$ one can follow the first part of the proof of the Main Lemma in \cite{mateu2009extra} where a stronger result is proved for kernels $K(x)=\frac{\omega(x)}{|x|^n}$ in $\R^n$ being $\omega$ homogeneous of degree zero and even. To adapt it to this case it is enough to recall that $\SU$ takes one of the following forms for $i,j=1,2$ and $i\neq j$: $$\frac{x_ix_j}{|x|^2}K_2(\varepsilon |x|),\; \frac{x_i^2-x_j^2}{|x|^3}K_1(\varepsilon|x|),$$ and now take as $\omega$ in the {aforementioned Main Lemma} either $\frac{x_ix_j}{|x|^2}$ or  $\frac{x_i^2-x_j^2}{|x|^2}$ and notice that by \eqref{eq:controlBessel} the rest of the expression behaves as $\frac{1}{|x|^2}.$
    \end{proof}

    Combining the results in this section with the uniqueness of weak solution gives finally the proof of Theorem \ref{th:vortex-patch}.

    \section{Convergence to the 2D Euler Equations} \label{sec:limit}
    
    {In this section, we study the behavior of solutions to \eqref{def:QGSW} as $\varepsilon \to 0$, 
    the regime in which the system formally reduces to the 2D Euler equations. 
    Physically, this corresponds to an infinite Rossby deformation radius, where free–surface effects cease to matter; 
    mathematically, it reflects the passage from the regularized operator $(\Delta - \varepsilon^2)^{-1}$ 
    to the classical Biot--Savart law $\Delta^{-1}$, thus connecting \eqref{def:QGSW} with standard Euler dynamics.} 

    Our convergence result in this direction is the following.
    
    \begin{theorem}\label{th:limitRegular}
        Let $q_0\in \LPet_c (\R^2,\R)$ for $0<\gamma<1.$ Then there exist unique solutions $q^\varepsilon(\cdot,t),$ $\omega(\cdot,t)$ in $\LPet_c(\mathbb{R}^2,\mathbb{R})$ to the \eqref{def:QGSW} equations and the Euler equation, respectively, with initial data \( q_0 \), {and a time $T>0$ such that}
\begin{equation} \label{eq:convgEuler}
    \lim_{\varepsilon\to 0} \sup_{t\in [0,T]} \Na{q^\varepsilon(t) - \omega(t)} = 0.
\end{equation}
    \end{theorem}

    Notice that although we proved existence and uniqueness of solutions with initial data on $\Lipa_c,$  our convergence result is not a convergence result in that space but instead  is the so-called little Hölder space, $\LPet_c(\R^2,\R)$, which is the closed subspace of $\Lipa_c(\R^2,\R)$ consisting of those functions satisfying the vanishing condition 
    \begin{equation*} 
        \lim_{h\to 0}\sup_{\substack{x,y\in\R^2\\ 0<|x-y|<h}} \frac{|f(x)-f(y)|}{|x-y|^\gamma}=0.
    \end{equation*}
    It is easy to see that $\LPet(\R^2,\R)$ is the closure of $\mathcal{C}^\infty(\R^2,\R)$ in the usual Hölder norm of $\Lipa(\R^2,\R)$. This fact is essential for our proof. 
    
    The existence and uniqueness result for the Euler equation {in the little Hölder space} can be seen in \cite{misiolek2018continuity} (see also \cite{magana2024continuitysolutionmapactive} for more general kernels). For the \eqref{def:QGSW} one can easily adapt the proofs found there, we give here a brief sketch.

    Consider as in Section \ref{sec:regular} the equation for the particle trajectories
    \begin{equation}\label{eq:mapsEps}
        \begin{aligned}
        \frac{dX^\varepsilon}{dt}(\alpha,t)
        &=\int_{\mathbb{R}^2} K^\varepsilon(X^\varepsilon(\alpha,t)-X^\varepsilon(\alpha',t))\,q_0(\alpha')\,d\alpha'
        \eqqcolon F^\varepsilon(X^\varepsilon(\alpha,t)),\\
        X^\varepsilon(\alpha,0) &= \alpha,
        \end{aligned}
    \end{equation} 

    {where $K^\varepsilon$ denotes the kernel defined in \eqref{def:kernel}, written to make the dependence on $\varepsilon$ explicit. For $\varepsilon=0$ one recovers the 2D Euler kernel $K^0(x)=\frac{1}{2\pi}\frac{x^\perp}{|x|^2}.$}
    
    We will apply Picard-Lindelöf theorem (Theorem \ref{th:Picard}) with $B=\LUPet(\R^2;\R^2)$ and the open set \begin{equation} \label{def:OpenSetU}
    U_\delta = \left\{ X: {\R^2} \to \R^2 : X= e+\varphi_X, \varphi_X\in\LUPet(\R^2;\R^2) \text{ and } \NUa{\varphi_X}<\delta \right\}, 
    \end{equation} where $e(x) = x$ and $0<\delta\leq \frac{1}{2}$ is chosen small enough so that 
    \begin{equation*}
        \inf_{x\in \R^2} \det DX(x)>\frac{1}{2}.
    \end{equation*}
    Clearly, $U_\delta$ can be identified with an open ball centered at the {identity} in $\LUPet(\R^2;\R^2)$ 
    {and each map $X\in U_\delta$ is bi-Lipschitz. Indeed, $DX(x)=I+D\varphi_X(x)$ satisfies $(1-\delta)|v|\leq |DX(x)v|\leq (1+\delta)|v|$ for all $v\in\R^2$. For $x,y\in\R^2$ we write $$X(x)-X(y)=\int_0^1 DX\big(y+t(x-y)\big)(x-y)\,dt,$$
    and applying the pointwise bounds under the integral yields
    $$
    (1-\delta)|x-y| \leq |X(x)-X(y)| \leq (1+\delta)|x-y|.$$
    } 
    
    {For $\delta$ small, $U_\delta$ is contained in the set $O_M$, for some $M > 1$, introduced in \eqref{def:obert} of Section~\ref{sec:regular}.}

    By using the explicit form of the distributional derivative of the kernel given in \eqref{eq:derDistK} and repeated use of Lemmas \ref{lemma:propietats} and \ref{lemma:SIO} one gets that for $X\in U_\delta$    
    \begin{equation}\label{eq:SNaGradF}
        \SNa{DF^\varepsilon(X)}\leq C_1(\Np{q_0}+\SNa{q_0})\SNa{D \varphi_X}+C_2(1+\Np{D \varphi_X})^{1+\gamma} \SNa{q_0},
    \end{equation}
    and     
    \begin{equation*}
        \NUa{F^\varepsilon(X)}\leq C{\Na{q_0}},
    \end{equation*}
    where the constants involved are independent of $\varepsilon$.
    
    Then $F^\varepsilon(X)$ maps $U_\delta$ to $\LUPet$ since \eqref{eq:SNaGradF} yield
    $$\lim_{h\to 0}\sup_{0<|x-y|<h}\frac{|D F^\varepsilon(X)(x)-D F^\varepsilon(X)(y)|}{|x-y|^\gamma}=0,$$
    since $\varphi_X\in \LUPet$ and $q_0\in\LPet.$ 

    Finally, to apply Picard-Lindelöf theorem one must see that $F^\varepsilon$ is a locally Lipschitz continuous mapping. This is indeed the case as shown in Section \ref{sec:regular}. {Moreover, we saw that the constants appearing are independent from $\varepsilon$. Therefore, there exists a time $T^{\text{QGSW}}$, independent of $\varepsilon$, such that the solution to \eqref{eq:mapsEps} is unique. The same Picard-Lindelöf argument applied to Euler yields a time $T^{\text{Euler}}>0$, for which the Euler particle trajectory map exists uniquely. Set $T\coloneqq \min \{T^{\text{QGSW}}, T^{\text{Euler}}\}$}.

    To prove \eqref{eq:convgEuler} we will need the following lemma, the proof of which can be seen in \cite[Lemma 2.2]{misiolek2018continuity}. 

    \begin{lemma} \label{lemma:contInversa}
        Let $0<\gamma<1.$ Suppose that $X,Y \in U_\delta.$ Then, the functions $X,Y \to X\circ Y$ and $X\to X^{-1}$ are continuous in the  $\LipUa$ norm topology.
    \end{lemma}

    Let $X^\varepsilon(t)$ and $X^0(t)$ be the corresponding Lagrangian flows solving the Cauchy problem in $U_\delta$, for the \eqref{def:QGSW} and Euler equations respectively, with initial data $q_0$. Given any $\eta >0$ and using he fact that smooth functions are dense in $\LPet,$ we can choose $\phi_\eta$ in $C_c^\infty(\R^2;\R)$ such that $$\Na{\phi_\eta-q_0}<\eta.$$
    
    We estimate
    \begin{equation*}
        \begin{split}
            \Na{q^{{\varepsilon}}-\omega}&=\Na{q_0\left((X^\varepsilon)^{-1}(\alpha,t)\right)-q_0\left((X^0)^{-1}(\alpha,t)\right)}\\
            &\leq \Na{q_0\left((X^\varepsilon)^{-1}(\alpha,t)\right)-\phi_\eta\left((X^\varepsilon)^{-1}(\alpha,t)\right)}\\ &+\Na{\phi_\eta\left((X^\varepsilon)^{-1}(\alpha,t)\right)-\phi_\eta\left((X^0)^{-1}(\alpha,t)\right)}\\&+\Na{\phi_\eta\left((X^0)^{-1}(\alpha,t)\right)-q_0\left((X^0)^{-1}(\alpha,t)\right)}=I+II+III.
        \end{split}
    \end{equation*}
    
   {Terms $I$ and $III$ are handled in the same way. Indeed, by \eqref{eq:NaCompInversa} we obtain $$ I \leq \Na{q_0-\phi_\eta}\bigl(1+\NUa{X^\varepsilon}^{3\gamma}\bigr) \leq C\eta, $$
and the same estimate holds for $III$.} {For the second term we have 
   \begin{equation*}
			\begin{aligned}
				II&\leq \int_0^1 \Na{\nabla \phi_\eta\left((1-r)(X^\varepsilon)^{-1}+r(X^0)^{-1}\right)\left((X^\varepsilon)^{-1}-(X^0)^{-1}\right)}dr\\
                &{\leq C \Na{\nabla \phi_\eta}\left(1+\NUa{X^\varepsilon}^{3\gamma}+\NUa{X^0}^{3\gamma}\right) \Na{(X^\varepsilon)^{-1}-(X^0)^{-1}}} \\
				& \leq C \NUa{ \phi_\eta} \NUa{(X^\varepsilon)^{-1}(\alpha,t)-(X^0)^{-1}(\alpha,t)}
			\end{aligned}
		\end{equation*}
    where on the second inequality we have used the algebra properties of Hölder functions, \eqref{eq:algebra}, and once more \eqref{eq:NaCompInversa}, and in the third inequality we estimate the $\Lipa$ norm of the difference by the $\LipUa$ norm — not an optimal bound, but sufficient for our purposes. } {Notice that although $\NUa{\phi_\eta}$ may grow as $\eta\to0$, it remains a fixed finite constant for each fixed $\eta>0$.} Therefore, if we can prove that $X^\varepsilon$ converges to $X^0$ in $\LUPet$ then by Lemma \ref{lemma:contInversa} we will have that $II$ converges to 0 and \eqref{eq:convgEuler} will follow.

    \begin{lemma}\label{lemma:FepsToF0}
        Let $0<\gamma<1$ and let $U_\delta$ be defined as in \eqref{eq:mapsEps}. Let $q_0\in~ \LPet_c(\R^2; \R)$.
        Then $$ \sup_{X\in U_\delta}\NUa{F^\varepsilon(X)-F^0(X)}\longrightarrow 0, \text{ as }\varepsilon\to 0.$$
    \end{lemma}

With the Lemma above, which we prove at the end of the section, we are now ready to prove the convergence of the flows. Subtracting the two equations and integrating in time, $$X^\varepsilon(\alpha,t)-X^0(\alpha,t)
=
\int_0^t
\bigl(F^\varepsilon(X^\varepsilon(\alpha,s))-F^0(X^0(\alpha,s))\bigr)\,ds .
$$

Adding and subtracting $F^\varepsilon(X^0(\alpha,s))$ inside the integrand and using that ${F^\varepsilon}$ is Lipschitz,
we obtain
$$
\NUa{X^\varepsilon(\alpha,t)-X^0(\alpha,t)}
\leq
L\int_0^t \NUa{X^\varepsilon(\alpha,s)-X^0(\alpha,s)}\,ds
+
t\sup_{X\in U_\delta}\NUa{F^\varepsilon(X)-F^0(X)} .
$$

Before applying the next estimate, we recall the following version of Grönwall’s lemma from \cite[Cor.~6.2]{Amann_1990}.

\begin{lemma}[Grönwall]
Let $J\subset\R$ be an interval and $t_{0}\in J$.  
Assume that $a,\beta,u\in C(J;\R^+)$ and that 
$a(t)=a_{0}(|t-t_{0}|),$ where $a_0\in C(\R^+;\R^+)$ is a monotone increasing function. Assume that
\[
u(t)\;\le\; a(t)
   + \Bigl|\int_{t_{0}}^{t}\beta(s)\,u(s)\,ds\Bigr|,
   \qquad \forall t\in J.
\]
Then
\[
u(t)\;\le\;
a(t)e^{\left|\int_{t_{0}}^{t}\beta(s)\,ds \right|},
\qquad \forall t\in J.
\]
\end{lemma}

The previous lemma yields
$$
\sup_{t\in[0,T]}\NUa{X^\varepsilon(\alpha,t)-X^0(\alpha,t)} \leq {T}e^{LT} \sup_{X\in U_\delta}\NUa{F^\varepsilon(X)-F^0(X)}
$$
which goes to 0 by Lemma \ref{lemma:FepsToF0} as $\varepsilon\to 0$. {Hence, $\forall t\in[0,T], \; X^\varepsilon \to X^0$ in $\LUPet(\R^2;\R^2).$} 

    \begin{proof}[Proof of Lemma \ref{lemma:FepsToF0}]
        We define first $$H(x)\coloneqq K^\varepsilon(x)-K^0(x)=\frac{\varepsilon x^\perp}{2\pi|x|}h(x), \quad h(x)\coloneqq K_1(\varepsilon|x|)-\frac{1}{\varepsilon |x|}.$$

        Notice that $H$ is not singular near the origin. In fact, using the series expansion of $K_1$, see \eqref{eq:reDefiKernels}, it's easy to see that for $\varepsilon |x|\leq 1,$ $|H|\leq C \varepsilon^2 |x|(|\ln{(\varepsilon|x|)| + 1)}$ and following Lemma \ref{lemma:propietatsK} we see that for $\varepsilon|x|\geq 1$, $|H|\leq \frac{\varepsilon}{2\pi}\left(Ce^{-\varepsilon|x|}+\frac{1}{\varepsilon |x|}\right).$

        Moreover, by using \eqref{propietat:besselDerivades} we have 
        $$\nabla H = \varepsilon\nabla \left(\frac{x^\perp}{2\pi |x|}\right) h(x)+\frac{\varepsilon x^\perp}{2\pi |x|}\left(-\frac{\varepsilon x}{|x|}K_0(\varepsilon |x|)-\frac{x}{|x|}h(x)\right),$$
        so proceeding as before we get, for $ \varepsilon |x| \leq 1$
        $$|\nabla H|\leq \varepsilon^2 C (|\ln(\varepsilon |x|)| +1), \quad |\nabla^2 H|\leq \frac{\varepsilon^2 C}{|x|}$$
        and for $\varepsilon|x|\geq 1$ {we will use the bounds $$|\nabla H|\leq C_1 \varepsilon, \quad |\nabla^2 H|\leq C_2 \varepsilon,$$} which follow directly from 
        $$|\nabla H|\leq \frac{\varepsilon}{2\pi}\left(Ce^{-\varepsilon|x|}+\frac{1}{\varepsilon |x|^2}\right), \quad |\nabla^2 H|\leq \frac{\varepsilon}{2\pi}\left(Ce^{-\varepsilon|x|}+\frac{1}{\varepsilon |x|^3}\right),$$
        whenever $\varepsilon|x|\geq 1$.
        
        Fix $X\in U_\delta$ and $\alpha\in\R^2.$ Let $S=\operatorname{supp} q_0\subset B_R(0)$, We define $$\N \coloneqq\{\alpha' \in S:\; \varepsilon|X(\alpha)-X(\alpha')|\leq 1 \},\qquad \F\coloneqq S\setminus \N.$$ 

        Now,  we split the difference into the regions just defined
        \begin{equation*}
            \begin{aligned}
                \bigl(F^\varepsilon(X)-F^0(X)\bigr)(\alpha)
            &=\int_{\R^2}H\left(X(\alpha)-X(\alpha')\right)\,q_0(\alpha')\,d\alpha'\\&=\int_{\N} H\left(X(\alpha)-X(\alpha')\right)\,q_0(\alpha')\,d\alpha'\\&+\int_{\F}H\left(X(\alpha)-X(\alpha')\right)\,q_0(\alpha')\,d\alpha' \eqqcolon I+II
            \end{aligned}
        \end{equation*}
        
        and using the bound for the difference of the kernels and that $q_0$ is compactly supported we have  
        \begin{align*}
        |I|     
        &\leq C \Np{q_0} \varepsilon \int_{\N\cap B_R(0)} \varepsilon |X(\alpha)-X(\alpha')|(|\ln \varepsilon|X(\alpha)-X(\alpha')||+1) 
              \,d\alpha'\leq  C \Np{q_0} \varepsilon R^2
        \end{align*}
        since $t|\ln{t}|+t$ is bounded on the considered domain, and  
        \begin{align*}
            |II|&\leq C\Np{q_0} \int_{\F \cap B_R(0)} \varepsilon\left(Ce^{-\varepsilon|X(\alpha)-X(\alpha')|}+\frac{1}{\varepsilon |X(\alpha)-X(\alpha')|}\right)d\alpha'\\
            &\leq C\Np{q_0} \varepsilon \int_{B_R(0)} \left(\frac{C}{e}+ 1\right)d\alpha'  
            \leq C\Np{q_0} \varepsilon R^2
        \end{align*}
        
        Therefore, $$\Np{F^\varepsilon(X)-F^0(X)}\leq C\Np{q_0} \varepsilon R^2.$$
        
        Differentiating under the integral gives
        $$
            \nabla(F^\varepsilon-F^0)(X)(\alpha)
            =\int_{\R^2}
                \nabla H \bigl(X(\alpha)-X(\alpha')\bigr)\,
               \nabla X(\alpha)\,q_0(\alpha')\,d\alpha',
        $$
        
        where no Dirac delta appears thanks to the cancellation of the singular parts of 
        $\nabla K^\varepsilon$ and $\nabla K^0$, see Proposition~\ref{prop:commutador}.

        Now, due to the bounds of $\nabla H$ one can proceed as before and split the integral into $\N$ and $\F$ so 
        \begin{align*}
            |\nabla(F^\varepsilon-F^0)(X)(\alpha)|
            &\leq \Np{q_0} \Np{\nabla X} \int_{S}
                |\nabla H(X(\alpha)-X(\alpha'))|\,
               d\alpha' \\ 
               &\leq \Np{q_0} \Np{\nabla X} \left(I+II\right)
        \end{align*}

        where 
        \begin{align*}
            I &= \int_{\N} |\nabla H(X(\alpha)-X(\alpha'))|\,
               d\alpha' \\&\leq \int_{\N \cap B_R(0)} \frac{\varepsilon}{|X(\alpha)-X(\alpha')|} \varepsilon|X(\alpha)-X(\alpha')||\ln{\varepsilon |X(\alpha)-X(\alpha')|}|\,
               d\alpha' \\
               &\leq C \varepsilon \int_{B_R(0)} \frac{1}{|\alpha-\alpha'|} \, d\alpha'  \leq C  \varepsilon R
        \end{align*}
        
        and 

        \begin{align*}
            II &= \int_{\F} |\nabla H(X(\alpha)-X(\alpha'))|\,
               d\alpha' \leq \int_{\F \cap B_R(0)} C_1 \varepsilon d\alpha' \leq C \varepsilon R^2
        \end{align*}

        so  $$\Np{\nabla (F^\varepsilon(X)-F^0(X))}\leq C\Np{q_0} \Np{\nabla X} \varepsilon (R+R^2).$$

        It remains to bound the Hölder seminorm of the gradient. Let $d=|\alpha-\beta|$ and write 
        \begin{align*}
            \nabla (F^\varepsilon - F^0)(X)(\alpha)&-\nabla(F^\varepsilon - F^0)(\beta) \\&= \int_{\R^2} (\nabla H(X(\alpha)-X(\alpha')) - \nabla H(X(\beta)-X(\alpha')))\nabla X(\alpha) q_0(\alpha') d\alpha' \\
            &-\int_{\R^2} \nabla  H(X(\beta)-X(\alpha'))(\nabla X(\alpha)-\nabla X(\beta))q_0(\alpha') d\alpha' \eqqcolon I + II.
        \end{align*}

        Now for $II$ we can proceed as for the $L^\infty$  norm
        \begin{align*}
            |II|&\leq d^\gamma \Na{\nabla X} \Np{q_0} \bigg(\int_{\N \cap B_R(0)} |\nabla H(X(\beta)-X(\alpha'))| d\alpha'  \\ 
            &\hspace{4cm}+\int_{\F\cap B_R(0)} |\nabla H(X(\beta)-X(\alpha'))| d\alpha'\bigg)\\ 
            &\leq C(R+R^2) d^\gamma \Np{q_0}\varepsilon 
        \end{align*}

        Then, {let $M > \frac{(1+\delta)}{(1-\delta)}$} and write, $I=I_1+I_2$ with 
        $$I_1= \int_{|\alpha-\alpha'|\leq Md} (\nabla H(X(\alpha)-X(\alpha')) - \nabla H(X(\beta)-X(\alpha')))\nabla X(\alpha) q_0(\alpha') d\alpha', $$
        $$I_2= \int_{|\alpha-\alpha'|\geq Md} (\nabla H(X(\alpha)-X(\alpha')) - \nabla H(X(\beta)-X(\alpha')))\nabla X(\alpha) q_0(\alpha') d\alpha'. $$

    For the first term we bound each part separately as before
    \begin{align*}
        |I_1|&\leq \int_{|\alpha-\alpha'|\leq Md} (|\nabla H(X(\alpha)-X(\alpha'))| + |\nabla H(X(\beta)-X(\alpha'))|)|\nabla X(\alpha)| |q_0(\alpha')| d\alpha'\\
        &\leq C \Np{q_0}\Np{\nabla X }  \bigg(\int_{\{|\alpha-\alpha'|\leq Md\}\cap \N} |\nabla H(X(\alpha)-X(\alpha'))|d\alpha' \\ &+\int_{\{|\alpha-\alpha'|\leq Md\}\cap \F} |\nabla H(X(\alpha)-X(\alpha'))|d\alpha' + \int_{\{|\alpha-\alpha'|\leq Md\}\cap \N[\beta]} |\nabla H(X(\beta)-X(\alpha'))|d\alpha' \\&+\int_{\{|\alpha-\alpha'|\leq Md\}\cap \F[\beta]} |\nabla H(X(\beta)-X(\alpha'))|d\alpha'\bigg) \\ &\leq C \Np{q_0}\varepsilon ( d+d^2).
    \end{align*}

    Finally, for $I_2$ we use the mean value inequality  to get
    \begin{align*}
        |I_2|&\leq d C \Np{\nabla X} \Np{q_0} \int_{\{|\alpha-\alpha'|\geq Md\}\cap B_R(0)} |\nabla^2 H(\xi(\alpha'))| d\alpha '
    \end{align*}
    where $|\xi(\alpha')| \asymp |\alpha-\alpha'|.$ Indeed, for some $\theta\in(0,1),$ we have $$\xi=X(\beta)-X(\alpha')+\theta(X(\alpha)-X(\beta))= X(\alpha)-X(\alpha')+(1-\theta)(X(\beta)-X(\alpha)),$$    
    so using that $X\in U_\delta$ is bi-Lipschitz we get
    \begin{align*}
         |\xi|&\geq |X(\alpha)-X(\alpha')|-|X(\beta)-X(\alpha)|\geq (1-\delta)|\alpha-\alpha'|-(1+\delta)d\\
         &\geq \frac{(1-\delta)M-(1+\delta)}{M}|\alpha-\alpha'|\eqqcolon c_1 |\alpha-\alpha'|
    \end{align*}
    and
    \begin{align*}
         |\xi|&\leq |X(\beta)-X(\alpha')|+|X(\alpha)-X(\beta)|\leq (1+\delta)[|\beta-\alpha'|+d]\\
         &\leq (1+\delta)[|\alpha-\alpha'|+2d]\leq \frac{(1+\delta)(M+2)}{M} |\alpha-\alpha'|\eqqcolon c_2 |\alpha-\alpha'|.
    \end{align*} 

    Therefore,     
    \begin{align*}
        |I_2|&\leq C d  \Np{q_0} \left(\int_{\substack{
    |\alpha-\alpha'|\ge Md \\
    \varepsilon|\xi(\alpha')|\le 1 \\
    \alpha'\in B_R(0)}} |\nabla^2 H(\xi(\alpha'))| d\alpha '+ \int_{\substack{
    |\alpha-\alpha'|\ge Md \\
    \varepsilon|\xi(\alpha')|\geq 1 \\
    \alpha'\in B_R(0)}} |\nabla^2 H(\xi(\alpha'))| d\alpha '\right)\\
    &\leq C d  \Np{q_0} \left(\int_{\substack{
    Md \leq |\alpha-\alpha'|\leq \frac{1}{\varepsilon c_1} \\    
    \alpha'\in B_R(0)}} \frac{\varepsilon^2}{|\xi(\alpha')|} d\alpha '+ \int_{B_R(0)} C_2 \varepsilon d\alpha '\right)\\
    &\leq C d  \Np{q_0} \left(\int_{\substack{
    0 \leq |\alpha-\alpha'|\leq \frac{1}{\varepsilon c_1} \\    
    \alpha'\in B_R(0)}} \frac{\varepsilon^2}{c_1|\alpha-\alpha'|} d\alpha '+ \int_{B_R(0)} C_2 \varepsilon d\alpha '\right)\\
        &\leq C(1+R^2) d\Np{q_0}\varepsilon.
    \end{align*}
    
    Gathering the previous bounds we get that $\Np{F^\varepsilon-F^0},$ $\Np{\nabla (F^\varepsilon-F^0)}$ and $\SNa{\nabla(F^\varepsilon-F^0)}$ all converge to $0$ when $\varepsilon\to0$ uniformly on $X\in U_{\delta}.$ Hence, $\NUa{F^\varepsilon(X)-F^0(X)}\to 0$ uniformly.
    \end{proof}

    \begin{remark}
        While this paper establishes the existence and boundary regularity of vortex patches for the \eqref{def:QGSW} equations (Theorem \ref{th:vortex-patch}), the convergence result of Theorem \ref{th:limitRegular} does not apply to vortex patches. The question of strong convergence of QGSW vortex patches to Euler vortex patches as $\varepsilon\to 0$ has been addressed recently by Houamed and the first author in \cite{Houamed_2025} where the Extrapolation Compactness method and Littlewood–Paley theory are used.
    \end{remark}

\section*{Funding and Acknowledgments}
This research is supported by the projects ``Análisis y Ecuaciones en Derivadas Parciales'' (Refs. PID2020-112881GB-I00 and PID2024-155320NB-I00) of the Spanish Ministry of Science, Innovation and Universities.

\bibliographystyle{plain} 
\bibliography{biblio}  
\end{document}